\documentclass[reqno]{amsart}
\usepackage[pagewise]{lineno}
\usepackage{amssymb}
\usepackage{enumerate}
\usepackage{bbm}
\usepackage{xcolor}
\usepackage{enumitem}
\usepackage{comment}
\usepackage[notref,notcite]{showkeys}
\allowdisplaybreaks

\begin{document}
	
	\newcommand{\dt}{\mathrm{d}t}
	\newcommand{\dx}{\mathrm{d}x}
	\newcommand{\dy}{\mathrm{d}y}
	\newcommand{\dxdt}{\mathrm{d}x\mathrm{d}t}
	\title[Sparse Optimal control for the VCHE]{Optimal control of the three-dimensional viscous Camassa-Holm equations with sparse controls}
	\author [ N. H. H. Giang]
	{Nguyen Hai Ha Giang$^{\natural1}$}
	
	\address{Nguyen Hai Ha Giang \hfill\break
		Department of Mathematics, Hanoi National University of Education \hfill\break
		136 Xuan Thuy, Cau Giay, Hanoi, Vietnam}
	\email{giangnhh@hnue.edu.vn} 
	
	\thanks{$^{\natural}$ Corresponding author: giangnhh@hnue.edu.vn}
	\numberwithin{equation}{section}
	\newtheorem{theorem}{Theorem}[section]
	\newtheorem{lemma}{Lemma}[section]
	\newtheorem{proposition}{Proposition}[section]
	\newtheorem{remark}{Remark}[section]
	\newtheorem{definition}{Definition}[section]
	\newtheorem{corollary}{Corollary}[section]
	\newtheorem{example}{Example}[section]
	\begin{abstract}
		We investigate a distributed optimal control problem for the viscous Camassa--Holm equations with sparse controls and a general cost functional. Considering three different forms of sparsity-promoting terms, we prove the existence of optimal solutions, derive the corresponding optimality conditions and analyze the stability of optimal solutions with respect to the sparsity parameter.
	\end{abstract}
	
	\maketitle
	\section{Introduction}
	
	Let $\Omega$ be a bounded domain in $\mathbb{R}^3$ with smooth boundary $\partial \Omega$ (at least $C^3$). We denote the time-space cylinder by $Q=\Omega\times (0,T)$.  In this work, we consider the minimization of the following objective functional
	\begin{align*}
		J(y,u)= \int_{0}^{T}g(t,y(t))dt + \frac{\gamma}{2} \iint_Q\vert u(x,t)\vert^2 \, dxdt + \kappa j(u), \hspace{10pt} \gamma>0.
	\end{align*}
	Here, we consider three choices of the functional $j(u)$:
	\begin{align}
		&j_1 (u): L^1(Q)^3 \rightarrow \mathbb{R}, j_1(u)=\Vert u\Vert_{L^1(Q)^3}= \iint_Q \vert u(x,t)\vert dxdt,\\
		&j_2(u): L^2(0,T;L^1(\Omega)^3) \rightarrow \mathbb{R}, j_2(u)=\Vert u \Vert_{L^2(0,T;L^1(\Omega)^3)}=\left[ \int_0^T \Vert u(t) \Vert_{L^1(\Omega)^3}^2 dt \right] ^{1/2},\\
		&j_3(u): L^1(\Omega;L^2(0,T)^3) \rightarrow \mathbb{R}, j_3(u)=\Vert u\Vert_{L^1(\Omega;L^2(0,T)^3)}=\int_\Omega \Vert u(x) \Vert_{L^2(0,T)^3} \,dx.
	\end{align}
	The corresponding objective functionals with $j_1, j_2, j_3$ will be denoted by $J_1, J_2, J_3,$ respectively. \\
	Meanwhile, the free variables, state $y$ and control $u$, must fulfill the following three-dimensional viscous Camassa-Holm equations \textcolor{red}{
		\begin{equation}\label{1.1}
			\begin{cases}
				\partial_{t}\left ( y-\alpha ^{2} \Delta y \right )+\nu  \left (Ay-\alpha ^{2} \Delta Ay  \right ) + \nabla p\\
				\hspace{70pt}= y \times \left ( \nabla\times\left ( y-\alpha ^{2}\Delta y \right ) \right ) + u,  &(x,t) \in (\Omega \times (0,+\infty)),\\
				\nabla \cdot y = 0, &(x,t) \in (\Omega \times (0,+\infty)), \\
				y = Ay = 0,   &(x,t) \in (\partial \Omega \times (0,+\infty)),  \\
				y(x,0) = y_0(x),  &\ x\in\Omega,
			\end{cases}
	\end{equation} }
	where $\alpha>0$, $\nu>0$ is the viscous constant, $u$ is the control, $y$ is the corresponding fluid velocity vector, $g:[0,T] \times V \to \mathbb{R}$ is a given function satisfying some certain assumptions which will be specified later in Section 2.\\
	Moreover, we require that the control $u$ belongs to the admissible set 
	\[U_{a, b}=\left \{ u\in L^2(0,T;L^2(\Omega)^3): a_i \leq u_i(x,t)\leq b_i \ \text{a.e} \ (x,t)\in Q, i=1,2,3 \right\},\]
	for real numbers $a_i\leq0<b_i$. \\
	Henceforth, the corresponding problems with $j_1, j_2, j_3$ will be denoted by $(\textbf{P}_1), (\textbf{P}_2),$ \\ $(\textbf{P}_3)$, respectively.
	
	The viscous Camassa-Holm equations (VCHE for short), also known as the Lagrangian averaged Navier-Stokes (LANS-$\alpha$) or Navier-Stokes-$\alpha$ equations, are a regularized modification of the Navier–Stokes equations that preserve the transport properties of circulation and vorticity dynamics. Initially, the inviscid ($\nu=0$) case, called the Lagrangian averaged Euler (LAE-$\alpha$) or Euler-$\alpha$ equations, was introduced in \cite{Holm} as a natural mathematical generalization of the integrable inviscid one-dimensional Camassa-Holm equation discovered in \cite{Camassa}. Then in \cite{Foias2002}, the authors incorporated viscous dissipation into the equations due to the physical significance of the momentum $u+\alpha^2Au$. Since then, mathematical problems related to the viscous Camassa-Holm equations have attracted considerable attention. In bounded domains with Dirichlet or periodic boundary conditions, many results have been established concerning the existence of solutions and global attractors, see e.g. \cite{Coutand, Foias2002, Ilyn2003, Kim2006, Marsden2001, Vishik2007} and references therein. The time decay rates of solutions on the whole space were extensively investigated in \cite{AT2017, schonbek2008, Gao2022, Zhao}. We also refer the interested readers to \cite{Albanez2016, AnhBach, Korn2009} for results on the data assimilation and to \cite{Gao2016, Mitra} for the controllability of the viscous Camassa-Holm equations. 
	
	In this paper, we focus on optimal control problems, particularly on sparse optimal control problems, for the mentioned three-dimensional viscous Camassa–Holm equations. Sparse controls refer to controllers that are nonzero only in a small region of the domain. Such problems arise from the need of finding suitable, small regions to place the controllers, rather than acting on the flow at every point of the domain. As stated in \cite{casas2017}, there are two approaches to obtain the sparsity effect, each with its own mathematical challenges and physical applicability. The first employs an $L^1$-norm of the control in the cost functional and the other uses controllers in measure spaces. Numerous results have been published concerning both directions, see \cite{casas2017,CHW2012,CHW2017,CK2013,CK2016,CK2011,KPV2014,PV2013} and the relevant references.
	
	Specifically, our problem contains an abstract cost functional and three choices of controls related to the $L^1$-norm.  Following the general approach in \cite{casas2017} or \cite{CHW2017}, we will prove the existence of optimal controls, establish the optimality conditions for this problem, but broaden by proving the stability of the control system with respect to the sparsity parameter. Compare with the existing results, the novelty of this article lies in two points. First, we allow the objective functional to be considerably more general. This class of functionals includes not only the standard quadratic form but also other classes as considered in \cite{AS2025}, while requiring even fewer conditions. Second, to the best of our knowledge, the problem of stability analysis of optimal solutions under variations of the sparsity parameter is so far unknown. There were works investigating the stability of solutions under different perturbations, e.g. \cite{AGS2025, bonnans2013perturbation, casas2015_2, CT2022}. However, these studies focus mainly on the cases of perturbations in the state equations, the initial data, or the Tikhonov regularizing parameter. The methods developed here, together with \cite{AGS2025}, can be further applied to establish the H\"older stability in the case of bang-bang-bang structure $\gamma=0$.
	
	The paper is organized as follows. In Section 2, we will recall some auxilliary and fundamental results on function spaces and inequalites for the nonlinear terms related to the Camassa--Holm equations, as well as the properties of weak solution and the control-to-state mapping. After proving the existence of optimal solutions, in Sections 4 and 5, we derive the optimality conditions. Finally, the Lipschitz and H\"older stability analysis of optimal controls with respect to the sparsity parameter are given in the last section.  
	
	\section{Preliminaries}
	\subsection{Function spaces and inequalities for the nonlinear terms}
	Throughout this article, we define
	\begin{align*}
		(u,v)&:=\int_{\Omega}\displaystyle\sum_{j=1}^{3}u_jv_j\,dx, u=(u_1,u_2,u_3),v=(v_1,v_2,v_3) \in L^2(\Omega)^3;\\
		((u,v))&:=\int_{\Omega}\displaystyle\sum_{j=1}^{3}\nabla u_j\cdot \nabla v_j\,dx, u=(u_1,u_2,u_3),v=(v_1,v_2,v_3) \in H_0^1(\Omega)^3,
	\end{align*}
	with the associated norms $|u|^2 :=(u,u), \|u\|^2:=((u,u))$, respectively. For convenience, we denote
	\begin{align*}
		\mathbb{L}^p(\Omega) &:= \left(L^p(\Omega)\right)^3; \\
		\mathbb{H}^s(\Omega) &:= \left(H^s(\Omega)\right)^3; \\
		\mathbb{H}^s_0(\Omega) &:= \left(H^s_0(\Omega)\right)^3.
	\end{align*}
	Let us define the following functional spaces
	\begin{align*}
		\mathcal{V} &:= \{u\in (C_{0}^{\infty}(\Omega))^3\mid  \nabla\cdot u=0\};\\
		H &:= \text{ the closure of $\mathcal{V}$ in }\mathbb{L}^p(\Omega);\\
		V &:= \text{ the closure of $\mathcal{V}$ in }\mathbb{H}^1_0(\Omega);\\
		\mathbb{H}^s_{\sigma} &:= \text{ the closure of $\mathcal{V}$ in }\mathbb{H}^s_0(\Omega);\\
		D(A)&:=\mathbb{H}^2(\Omega)\cap V.
	\end{align*}
	The spaces $H$ and $V$ are Hilbert spaces with scalar products $(\cdot,\cdot)$ and $((\cdot,\cdot))$, respectively. It follows that $D(A) \subset V \subset H \equiv H'\subset V'$, where the injections are dense and continuous. Moreover, the imbeddings $V \hookrightarrow H$ and $D(A) \hookrightarrow V$ are compact. \\
	Denote by $A$ the Stokes operator, with domain $D(A)=\mathbb{H}^2(\Omega)\cap V$, defined by $Au=-\mathcal{P}(\Delta u)$, where $\mathcal{P}$ is the Leray projection, i.e. the projection operator from $\mathbb{L}^2(\Omega)$ onto $H$.
	Since $A$ is closed operator with discrete spectrum, we can define the fractional powers of $A$, namely $A^\alpha$, where $\alpha>0$ as follows: we already know that the set of eigenfunctions $\{w_i\}$ of $A$ form an orthornormal basis of $H$. For each $u \in H$, define
	\begin{align*}
		A^\alpha u = \sum_{n=1}^\infty \lambda^{\alpha}_n(u,w_n)w_n,
	\end{align*}
	where $\lambda_n$ is the eigenvalue with respect to the eigenfunction $w_n$. Then the domain of $A^\alpha$ in $H$ is
	\begin{align*}
		D(A^\alpha) = \{u: |A^\alpha u|<\infty\}.
	\end{align*}
	The space $D(A^\alpha)$ is a Banach space with the associated norm $\vert A^\alpha u \vert$. Moreover, we can define a scalar product on $D(A^\alpha)$ to make it a Hilbert space 
	\begin{align*}
		(u,v)_{D(A^{\alpha})} = (A^\alpha u, A^\alpha v).
	\end{align*}
	When $\alpha$ is equal to $k/2$, $k \in \mathbb{Z}^+$, it is well-known that if $k$ is even, we have the inclusion $\mathbb{H}^k(\Omega) \cap \mathbb{H}^{k-1}_0 \subset D(A^{k/2})$; if $k$ is odd then $\mathbb{H}^k_0 \subset D(A^{k/2})$. Moreover, if the boundary $\partial\Omega$ is of class $C^k$, then the norm $|A^{k/2}u|$ is equivalent to the $\mathbb{H}^k$-norm, namely there exists $c_1,c_2>0$ such that
	\begin{align*}
		c_1|A^{k/2}u|\leq \|u\|_{\mathbb{H}^k}\leq c_2|A^{k/2}u|.
	\end{align*}
	For more facts concerning the fractional powers of $A$, see e.g. \cite{robinson2003}. In particular, we mainly focus on the case when $k=1,2,3$. For the case $k=1$, we have the $D(A^{1/2})$ norm is equivalent to the $\mathbb{H}^1_0$-norm. For the case $k=2$, we receive the familiar $D(A)$ space. For the case $k=3$, as mentioned above, we have the inclusion $\mathbb{H}^3_\sigma(\Omega) \subset D(A^{3/2})$ and the norm $|A^{3/2}u|$ is equivalent to the $\mathbb{H}^3$-norm, or further the $\mathbb{H}^3_{\sigma}$-norm. Additionally, from the case $k=1$ and $k=3$, we can define the norm $\|Au\|$ which is equivalent to $|A^{3/2}u|$ and $\|u\|_{\mathbb{H}^3_\sigma}$.\\
	When $u$ belongs to $\mathbb{H}^3_{\sigma}$, we can define an operator $A^2$ with value in $V'$ such that
	\begin{align*}
		\langle A^2u,v\rangle_{V',V} = (A^{3/2}u,A^{1/2}v) = ((Au,v))
	\end{align*}
	for every $v\in V$. Similarly, one can also define the operator $A^2$ with value in $D(A)'$ as
	\begin{align*}
		\langle A^2u,v\rangle_{D(A)',D(A)} = (Au,Av)
	\end{align*}
	for every $v\in D(A)$.\\
	We denote by $L^p(0,T;X)$ the standard Banach space of all functions from $(0,T)$ to a real Banach space $X$, endowed with the norm
	\begin{align*}  \|y\|_{L^p(0,T;X)} &:=\left(\int_{0}^{T}\|y(t)\|_{X}^{p}\,dt\right)^{1/p};\\
		\|y\|_{L^{\infty}(0,T;X)} &:= \mathrm{ess sup} \{\|y(t)\|_X\mid t\in (0,T)\}.
	\end{align*}
	In particular,  $L^2(0,T;V)$ is a Hilbert space with the dual space $L^2(0,T;V')$. The duality pairing between $u \in L^2(0,T;V')$ and $ v \in L^2(0,T;V)$ is
	\begin{equation*}
		\langle u,v\rangle_{L^2(V'),L^2(V)}=\int_{0}^{T}\langle u(t),v(t)\rangle_{V',V}\,dt.
	\end{equation*}
	For convenience, let us define the space of abstract functions  whose time derivative lies in $L^2(0,T;V)$ as
	\begin{align*}
		W(0,T;H^3,V):=\left\{y \in L^2(0,T;\mathbb{H}^3_{\sigma})\mid \dfrac{dy}{dt} \in L^2(0,T;V)\right\},
	\end{align*}
	associated with the norm $\|\cdot\|_{W(0,T;H^3,V)}$ defined by
	\begin{align*}
		\|y\|^2_{W(0,T;H^3,V)}:=\|y\|^2_{L^2(0,T;\mathbb{H}^3_{\sigma})} + \|y_t\|^2_{L^2(0,T;V)}.
	\end{align*}
	By Corollary 7.3 in \cite{robinson2003}, the space $W(0,T;H^3,V)$ is continuously embedded in $L^{\infty}(0,T;\mathbb{H}^2_{\sigma})$. Moreover, since $
	\Omega$ is $C^3$, the embedding $W(0,T;H^3,V) \hookrightarrow L^2(0,T;D(A))$ is compact.\\
	We define the trilinear form $b:V \times V\times V \to \mathbb{R}$ by
	\begin{equation}
		b(u,v,w):=\displaystyle\sum_{i,j=1}^{3}\int_{\Omega}u_i \dfrac{\partial v_j}{\partial x_i} w_j \,dx.
	\end{equation}
	One can verify that if $u,v,w \in V$ then\\
	\centerline{$b(u,v,w)=-b(u,w,v)$.}
	Hence $b(u,v,v)=0 \ \forall u,v \in V.$
	\begin{lemma}
		[\cite{temam1995navier},\cite{robinson2003}]\label{lem2.1}
		If $n=3$ then the following estimates hold
		\begin{enumerate}
			\item $|b(u,v,w)| \leq C|u|^{1/4}\|u\|^{3/4} \|v\| |w|^{1/4}\|w\|^{3/4}$ for all $(u,v,w)\in V\times V \times V$;
			\item $|b(u,v,w)| \leq C\|u\| \|v\| \|w\|$ for all $(u,v,w)\in V\times V \times V$;
			\item $|b(u,v,w)| \leq C|u| \|v\||Aw|$ for all $(u,v,w)\in H \times V \times D(A)$;
			\item $|b(u,v,w)| \leq C|Au| \|v\| |w|$ for all $(u,v,w)\in D(A) \times V \times H$.
		\end{enumerate}
	\end{lemma}
	Furthermore, we can define a continuous bilinear operator $ \widetilde{B}$ from $ V \times V $ into $ V'$ by \[ \left \langle \widetilde{B}(u,v),w \right \rangle _{V',V} = \widetilde{b}(u,v,w), \]
	where $ \widetilde{b}(u,v,w) = b(u,v,w) - b(w,v,u). $
	The following results hold for $ \widetilde{B}$
	\begin{align*}
		(i)\ & \left \langle \widetilde{B}(u,v),w \right \rangle_{V',V} = - \left \langle \widetilde{B}(w,v),u \right \rangle_{V',V}, \\
		(ii)\ &  \left \langle \widetilde{B}(u,v),u \right \rangle_{V',V} \equiv 0,
	\end{align*}
	for every  $ u,v,w \in V $.
	
	\subsection{Existence and uniqueness of weak solutions to the viscous Camassa - Holm equations}
	
	\begin{definition}\label{VCHE_weaksolution}
		Let $ u \in L^2(0,T;\mathbb{L}^2(\Omega)) $ and let $ T > 0 $. A function \[ y \in L^\infty(0,T;\mathbb{H}_{\sigma}^2(\Omega)) \cap L^2(0,T;\mathbb{H}_{\sigma}^3(\Omega)) \] with $\displaystyle \frac{dy}{dt} \in L^2(0,T;V)  $ is said to be a weak solution to problem \eqref{1.1} on the interval $ (0,T)  $ if it satisfies
		\begin{align} \label{weak solution}
			\left \langle \partial_{t}y + \alpha^2 A\partial_{t}y, w \right \rangle_{V', V} &+ \nu \left \langle A(y + \alpha^2 Ay), w \right \rangle_{V', V} \nonumber \\
			& + \left \langle \widetilde{B}(y, y+\alpha^2 Ay), w \right \rangle_{V', V} = \langle u,w \rangle_{V',V}
		\end{align}
		for every $ w \in V $ and for almost every $ t \in [0,T] $. Moreover, $ y(0) = y_0 $ in $ D(A) $.
	\end{definition}
	Here, equation \eqref{weak solution} is understood in the following sense: for a.e. $ t_0, t \in [0,T] $ and for all $ w \in V $, we have
	\begin{align*}
		(y(t) &+ \alpha^2 Ay(t),w) - (y(t_0) + \alpha^2 Ay(t_0),w) + \nu \int_{t_0}^{t} (y(s) + \alpha^2 Ay(s),Aw) ds \\
		&+ \int_{t_0}^{t} \left \langle \widetilde{B}(y(s),y(s) + \alpha^2 Ay(s)),w \right \rangle {_{V', V}} ds = \int_{t_0}^{t} (u,w) ds.
	\end{align*}
	The following theorem is proved by using the arguments in \cite{schonbek2008} (or \cite{Coutand}).
	\begin{theorem}\label{existence_uniqueness}
		For $ u \in L^2(0,T;\mathbb{L}^2(\Omega)), T >0 $ and $ y_0 \in D(A) $ given, there exists a unique weak solution to problem \eqref{1.1} on the interval $ (0, T) $ in the sense of Definition \ref{VCHE_weaksolution}. Moreover, there exists a constant $ C $ such that the function $ u $ satisfies the following estimate for all $ t \in [0, T] $,
		\begin{align*}
			|y(t)|^2 &+\alpha^2 \|y(t)\|^2+\nu \int_{0}^{t}(\|y(s)\|^2+\alpha^2 |Ay(s)|^2)ds \\ & \leq |y_0|^2 + \alpha^2 \|y_0\|_{D(A)}^2+ C \|u\|^2_{L^2(0,T;\mathbb{L}^2(\Omega))}. 
		\end{align*}
	\end{theorem}
	\begin{remark}
		{\rm The boundedness of $ y $ in $ L^{\infty}(Q)^3 $ can be shown if $ u $ is bounded in $ L^2(0,T;\mathbb{L}^2(\Omega)) $. }
	\end{remark}
	
	\subsection{The control-to-state mapping}
	Let us study the behavior of the mapping: right-hand side $ \mapsto $ solution, the so-called \textit{control-to-state mapping} and its Fréchet differentiability.
	\begin{definition} \label{Smapping}
		Consider the system \eqref{1.1}. The control-to-state mapping $ u \mapsto y $, where $ y $ is the weak solution of \eqref{1.1} with fixed initial value $ y_0$, is denoted by $ S $, i.e. $ y = S(u) $.
	\end{definition}
	Let us recall the following known results. Taking into account that we refer the interested readers to \cite{AG2024} for the proofs of these lemma and theorems.
	\begin{lemma}\label{LipchitzS}
		Let $y_1,y_2 \in W(0,T;H^3,V)$ be the weak solutions of \eqref{1.1} with the right-hand sides of the first equations equal to $u_1$ and $u_2$ in $L^2(Q)^3$ and the initial data equal to $y_{01}$ and $y_{02}$, respectively. Then there exists a constant $C$ such that the following estimate holds
		\begin{align}
			\| y_1 - y_2 \|^2_{W(0,T;H^3,V)} \leq C\left(\|y_{01} - y_{02}\|_{D(A)}^2 + \| u_1 - u_2\|_{L^2(Q)^3}^2 \right).
		\end{align}
	\end{lemma}
	
	\begin{theorem}\label{1stS}
		The control-to-state mapping is Fréchet differentiable as mapping from $ L^2(0,T;\mathbb{L}^2(\Omega)) $ to $ W(0,T;H^3,V) $. The derivative at $ \overline{u} \in L^2(0,T;\mathbb{L}^2(\Omega)) $ in direction $ k \in L^2(0,T;\mathbb{L}^2(\Omega)) $ is given by $ S'(\overline{u} )k = y $, where $ y $ is the weak solution of 
		\begin{equation} \label{linearized eq}
			\begin{cases}
				\partial_{t} (y + \alpha^{2} Ay) + \nu A(y + \alpha^{2} Ay) + \widetilde{B}'(\overline{y}, \overline{y}+\alpha^2 A \overline{y})y = k \ \text{in} \ L^2(0,T;V'),\\
				y(0)=0 \ \text{in} \ D(A),
			\end{cases}
		\end{equation}
		with $ \overline{y} = S(\overline{u}) $ and $ \widetilde{B}'(\overline{y}, \overline{y}+\alpha^2 A \overline{y})y = \widetilde{B}(y, \overline{y} + \alpha^{2} A \overline{y}) + \widetilde{B}(\overline{y}, y + \alpha^{2}A y) $.
	\end{theorem}
	
	\begin{theorem}\label{2ndS}
		The control-to-state mapping is twice continuously differentiable as mapping from $ L^2(0,T;\mathbb{L}^2(\Omega)) $ to $ W(0,T;H^3,V) $. The derivative at $ \overline{u} \in L^2(0,T;\mathbb{L}^2(\Omega))$ in directions $ u_1, u_2 \in L^2(0,T;\mathbb{L}^2(\Omega))$ is given by $ S''(\overline{u})(u_1,u_2) = y $, where $ y $ is a weak solution of
		\begin{equation} \label{2nd_S}
			\begin{cases}
				\partial_{t} (y + \alpha^{2} Ay) + \nu A(y + \alpha^{2} Ay) + \widetilde{B}'(\overline{y}, \overline{y}+\alpha^2 A \overline{y})y = -\widetilde{B}''(\overline{y},\overline{y}+\alpha^2 A\overline{y})(y_1,y_2) ,\\
				y(0)=0 \ \text{in} \ D(A),
			\end{cases}
		\end{equation}
		with $ \overline{y} = S(\overline{u}) $ and $ y_i = S'(\overline{u})u_i, i = 1,2 $.\\
		Here, the term $ \widetilde{B}''(\overline{y},\overline{y}+\alpha^2 A\overline{y})(y_1,y_2)  $ is decomposed as
		\[ \widetilde{B}''(\overline{y},\overline{y}+\alpha^2 A\overline{y})(y_1,y_2) = \widetilde{B}(y_1, y_2 + \alpha^2 Ay_2) + \widetilde{B}(y_2, y_1 + \alpha^2 A y_1). \]
	\end{theorem}
	
	\subsection{Assumptions on the objective functional and the set of admissible control}
	
	Recall that the objective functional of the so-called original optimal control problem is
	\begin{align}
		J(y,u) = \int_0^T g(t,y(t))\dt + \frac{\gamma}{2} \iint_Q \vert u(x,t) \vert^2 dxdt. \label{2.44}
	\end{align}
	
	We shall impose the following assumptions on the functions $g$.
	
	\begin{itemize}
		\item[(\textbf{A1})]\label{A1} The function $g: [0,T] \times V
		\rightarrow \mathbb{R}^+$ is strictly convex, twice continuously Fr\'echet differentiable on $V$. Furthermore, for every $r>0$, there exists $L_r>0$ independent of $t$ such that
		\begin{align*}
			|g(t,y_1)-g(t,y_2)|\leq L_r\|y_1-y_2\| , \quad \forall \, \|y_1\|+\|y_2\|\leq r,t\in [0,T].
		\end{align*}
	\end{itemize}
	
	We now give some examples of the objective functional.
	\begin{example}{\rm 
			A first typical choice for the functional $J$ is the quadratic functional given by 
			\begin{align}
				J(y,u) = \dfrac{\alpha_Q}{2}\iint_Q |y(x,t) - y_Q(x,t)|^2 \,dxdt \nonumber +\dfrac{\gamma}{2} \iint_Q |u(x,t)|^2\,dxdt,  \label{2.46}
			\end{align}
			where $\alpha_Q,\gamma \in \mathbb{R}$, $\gamma>0$, and the function $y_Q \in L^2(Q)^3$ is the desired state. \\
			Additionally, the first term in our objective functional $J$ can be replaced by the integral $\displaystyle \dfrac{\alpha_Q}{2}\iint_Q |y(x,t) - y_Q(x,t)|^p\,dxdt $ with $2\leq p <6$.
		}
	\end{example}
	
	\begin{example}{\rm 
			Beside the above quadratic functional,  another notable choice for the functional $J(y,u)$ is
			\begin{equation*}
				J(y,u)= \iint_Q |\nabla y(x,t) - \nabla y_d(x,t)|^2 \,dxdt + \textcolor{red}{\dfrac{\gamma}{2}} \iint_Q |u(x,t)|^2\,dxdt,
			\end{equation*}
			where $y_d \in L^2(0,T;D(A))$ is the desired state. }
	\end{example}
	
	\section{Existence of optimal solutions}
	
	In this section we will show that the optimal control problems $(\textbf{P}_k), k=1,2,3$ admit a global solution. To do this,  we need the following lemma.
	\begin{lemma} \label{lem3.1}
		If ${y_n}$ converges to $y$ in $W(0,T;H^3,V)$ weakly, then for any $v \in L^2(0,T;V)$, we get the following convergence
		\begin{align*}
			\langle \widetilde{B}(y_n, y_n+\alpha^2 Ay_n),v \rangle _{L^2(0,T;V'),L^2(0,T;V)} \rightarrow \langle \widetilde{B}(y,y+\alpha^2Ay),v \rangle _{L^2(0,T;V'),L^2(0,T;V)}
		\end{align*} as $n \rightarrow \infty$.
	\end{lemma}
	
	\begin{proof}
		Consider the difference $\langle \widetilde{B}(y_n, y_n+\alpha^2 Ay_n),v \rangle _{L^2(0,T;V'),L^2(0,T;V)} - \langle \widetilde{B}(y,y+\alpha^2Ay),v \rangle $, we deduce that
		\begin{align*}
			&\left| \langle \widetilde{B}(y_n, y_n+\alpha^2 Ay_n),v \rangle _{L^2(0,T;V'),L^2(0,T;V)} - \langle \widetilde{B}(y,y+\alpha^2 Ay),v \rangle_{L^2(0,T;V'),L^2(0,T;V)}  \right|\\
			&\leq \left| \langle \widetilde{B}(y_n-y,y+\alpha^2Ay), v \rangle_{L^2(0,T;V'),L^2(0,T;V)}  \right| \\
			&\hspace{5pt}+ \left| \langle \widetilde{B}(y_n, y_n- y +\alpha^2 A(y_n-y),v \rangle_{L^2(0,T;V'),L^2(0,T;V)}  \right|\\
			&\leq \int_0^T \left| b(y_n-y, y+\alpha^2Ay, v) \right| dt + \int_0^T \left| b(v, y+\alpha^2Ay, y_n-y) \right| dt \\
			&\hspace{5pt}+\int_0^T \left| b(y_n, y_n-y+\alpha^2 A(y_n-y),v) \right| dt + \int_0^T \left| b(v, y_n-y+\alpha^2 A(y_n-y),y_n) \right| dt \\
			&\leq C \int_0^T \Vert y_n - y \Vert_{D(A)} \vert y+\alpha^2Ay \vert \Vert v \Vert \ dt \\
			&\hspace{5pt}+ C \int_0^T \Vert y_n \Vert_{D(A)} \vert y_n - y+ \alpha^2 A(y_n - y) \vert \Vert v \Vert \ dt \\
			&\leq C \int_0^T \Vert y_n - y \Vert_{D(A)} \Vert y \Vert_{D(A)} \Vert v \Vert dt + C \int_0^T \Vert y_n \Vert_{D(A)} \Vert y_n - y \Vert_{D(A)} \Vert v \Vert dt \\
			&\leq C \Vert y_n - y \Vert_{L^2(0,T;D(A)} \Vert y \Vert_{L^\infty(0,T;D(A))} \Vert v \Vert_{L^2(0,T;V)} \\
			&\hspace{5pt}+ C \Vert y_n - y \Vert_{L^2(0,T;D(A))}  \Vert y_n \Vert_{L^\infty(0,T;D(A))} \Vert v \Vert_{L^2(0,T;V)}.   
		\end{align*}
		Since $y$ is bounded in $L^\infty(0,T;D(A))$, $\{y_n\}$ is bounded in $L^\infty(0,T;D(A))$ and $y_n \rightarrow y $ in $L^2(0,T;D(A))$. We conclude that \[ \langle \widetilde{B}(y_n, y_n+\alpha^2 Ay_n),v \rangle _{L^2(0,T;V'),L^2(0,T;V)} \rightarrow \langle \widetilde{B}(y,y+\alpha^2Ay),v \rangle _{L^2(0,T;V'),L^2(0,T;V)} \] as $n \rightarrow \infty$.
	\end{proof}
	
	\begin{theorem}\label{theo3.1}
		Problems $(\textbf{P}_k)$  admit at least a globally optimal solution $\overline{u} \in U_{a,b}$ with associated state $\overline{y} \in W(0,T;H^3,V)$  for every $k=1,2,3$.
	\end{theorem}
	
	\begin{proof}
		For an arbitrary control $u$ in $L^2(Q)^3$, by Theorem \ref{existence_uniqueness},  equations \eqref{1.1} has a unique weak solution $y \in W(0,T;H^3,V)$. Furthermore, since $J_k(y,u) = J(y,u) +j_k(y)$, there exists the infimum of $J_k$ over all admissible controls and states
		\begin{align*}
			0\leq \overline{J}_k:= \text{inf } J_k(y,u) < \infty.
		\end{align*}
		Moreover, there is a sequence $\{(y_n,u_n)\}$ of admissible pairs such that $J_k(y_n,u_n)$ converges to $\overline{J}_k$ as $n \rightarrow \infty$. 
		
		From there, the proof follows a standard technique for optimal control of PDEs. We present the main steps here for completeness only. 
		
		\textbf{Step 1:} The sequence $\{u_n\}$ is bounded in $L^2(Q)^3$ and $\{y_n\}$ is bounded in $W(0,T;H^3,V)$.
		
		\textbf{Step 2:} There exists a subsequence $(y_{n'},u_{n'})$ converging to $(\overline{y},\overline{u})$ weakly in $W(0,T;H^3,V)\times L^2(Q)^3$, which is an admissible pair.
		
		\textbf{Step 3}: We will show that $\overline{J}_k=J_k(\overline{y},\overline{u})$ for every $k=1,2,3$. 
		Since $j_1,j_2,j_3$ and $g$ are convex and continuous, each $j_k+g$ is consequently weakly lower semicontinuous for all $k=1,2,3$. Moreover, we get that $y_n$ converges to $\overline{y}$ weakly in $W(0,T;H^3,V)$ and $W(0,T;H^3,V)$ is continuously embedded in $C([0,T];V)$, which infer that $y_n(t)$ converges weakly to $\overline{y} (t)$ in $V$. Hence, $g(t,\overline{y}(t))\leq \underset{n'\rightarrow\infty}{\liminf} g(t, y_{n'}(t))$ and we get
		\begin{equation*}
			\int_{0}^{T} g(t,\overline{y}(t))\,dt \leq\displaystyle\int_{0}^{T} {\liminf_{n'\to \infty}} g(t,y_{n'}(t))\,dt \leq {\liminf_{n'\to\infty}} \int_0^Tg(t,y_{n'}(t))\,dt,
		\end{equation*}
		where the last inequality is deduced by using Fatou's lemma.
		Similar expressions can be proved for the norm functionals $j_1,j_2,j_3$. \\
		Therefore
		\begin{align*}
			J_k(\overline{y},\overline{u}) &= \displaystyle\int_{0}^{T} g(t,\overline{y}(t))\,dt +\frac{\gamma}{2}\iint_Q \vert u(x,t) \vert^2 \, dxdt + \kappa j_k (\overline{y}) \\ &\leq {\liminf_{n' \to \infty}} \displaystyle\int_{0}^{T} g(t,y_{n'}(t))\,dt+\frac{\gamma}{2}{\liminf_{n'\to \infty}} \displaystyle\iint_Q  \vert u_{n'}(x,t) \vert^2 \, dxdt + \kappa \,{\liminf_{n'\to \infty}} j_k(y_{n'}) \\ &\leq {\liminf_{n' \to \infty}} J_k(y_{n'},u_{n'})=\overline{J}_k.
		\end{align*}
		Since $(\overline{y},\overline{u}) $ is admissible, and $\overline{J}_k $ is the infimum over all admissible pairs, it follows that $\overline 
		{J}_k=J_k(\overline{y},\overline{u})$. The proof is complete.
	\end{proof}
	\section{First-order necessary optimality conditions}
	Given that $L^2(Q)^3$ is a Hilbert space and $U_{a,b}$ is a convex subset of $L^2(Q)^3$, we denote by $\mathcal{N}_{U_{a,b}}(u), \mathcal{T}_{U_{a,b}}(u)$
	the normal cone and the polar cone of tangents of $U_{a,b}$ at the point $u \in U_{a,b}$ respectively, i.e.
	\begin{align*}
		\mathcal{N}_{U_{a,b}}(u)&=\{z \in L^2(Q)^3: (z,v-u) \leq 0, \forall v \in U_{a,b}\}\\
		\mathcal{T}_{U_{a,b}}(u)&=\{z\in L^2(Q)^3: (z,v) \leq 0, \forall v\in \mathcal{N}_{U_{a,b}}(u)\}.
	\end{align*}
	
	Moreover, from Theorems \ref{1stS} and \ref{2ndS}, we see that the composite functional $\mathcal{J}(u)=J(S(u),u)$ is also of class $C^2$. Henceforth, instead of the two-variable functional $J(y,u)$, we will work with the reduced functional $\mathcal{J}(u)=J(S(u),u)$. The complete original optimal control problems can be reduced to
	\begin{center}
		$(\textbf{P}_k)$ \hspace{8pt} find $\min \mathcal{J}_k(u), k=1,2,3$ 
	\end{center}
	subject to the control constraint
	\begin{align*}
		u \in U_{a,b}.
	\end{align*}
	\begin{definition}\label{def4.1}
		A control $\overline{u}$ is said to be locally optimal (in $L^2$-sense) if there exists a constant $\rho >0$ such that
		\begin{equation} \label{4.1}
			\mathcal{J}(\overline{u})\leq \mathcal{J}(u)
		\end{equation}
		holds for all $u \in U_{a,b}$ with $\|u-\overline{u}\|_{L^2(Q)^3}\leq \rho$. If the inequality \eqref{4.1} is strict at every $u\ne \overline{u}$, $\overline{u}$ is called a strict local solution.
	\end{definition}
	
	For fixed $t\in [0,T]$, the first-order derivative of $g(t,\cdot)$ with respect to $y$ in direction $w \in V$ is $g'_y(t,y) \circ w$. We can also consider $g_y'(t, y)$ as an element of $V'$, then the duality pairing of $V'$ and $V$ is the continuation of the $H-$scalar product to $V' \times V$. It is denoted by $\langle\cdot,\cdot\rangle_{V',V}$, we write
	\begin{align*}
		\langle g_y'(t,y), w \rangle_{V',V}=\int_{\Omega}g_y'(t,y)w(x)\,dx.
	\end{align*}
	Before derive the first-order necessary conditions for problem $(\textbf{P})$ with the objective functional $\mathcal{J}$, due to the complexity of the nonlinear term $\widetilde{B}^*$, let us consider separately the adjoint equations.
	
	\begin{lemma}\label{adj_wellposedness}
		Let $\overline{u}$ be a locally optimal control with associated state $\overline{y}$. Then there exists a unique $\overline{\lambda} \in W(0,T;H^3,V)$, which is the weak solution of the equations
		\begin{align}\label{adj_eqt}
			\begin{split}
				\begin{cases}
					- \partial_t(\overline{\lambda}+\alpha^2A\overline{\lambda})+\nu A (\overline{\lambda}+ \alpha^2 A \overline{\lambda}) +\widetilde{B}'(\overline{y},y+\alpha^2A\overline{y})^*\overline{\lambda} = \textcolor{red}{g'_y(t,\overline{y}(t))} \\ \hspace{220pt}\text{ in }L^2(0,T;D(A)'),\\
					\overline{\lambda}(T) + \alpha^2 \textcolor{red}{A}\overline{\lambda}(T) = 0
				\end{cases}
			\end{split}
		\end{align}
		where $\left \langle \widetilde{B}'(\overline{y},\overline{y}+\alpha^2A\overline{y})^*\overline{\lambda}, w \right \rangle_{D(A)',D(A)}:=\tilde{b}(\overline{y},w+\alpha^2Aw,\overline{\lambda})+\tilde{b}(w,\overline{y}+\alpha^2A\overline{y},\overline{\lambda}).$
	\end{lemma}
	\begin{proof}
		We will show that the linear adjoint equations \eqref{adj_eqt} possess a weak solution $\overline{\lambda}$ that belongs to $W(0,T;H^3,V)$.
		
		We define $W_0$ as a closed linear subspace of $W(0,T;H^3,V)$ by
		\begin{equation*}
			W_0=\{y\in W(0,T;H^3,V): y(0)=0\}.
		\end{equation*}
		Since the operator $S'(\overline{u}): L^2(0,T;D(A)') \to W(0,T;H^3,V)$ is an isomorphism, for all $z \in W_0$, there exists an unique element $m \in L^{2}(0,T;D(A)')$ such that $z=S'(\overline{u})m$. Take $g \in W_0'$, denote $\overline{\lambda} = S'(\overline{u})^*g\in L^2(0,T;D(A))$, then we have the actions of adjoint operator $S'(\overline{u})^*$ as follow
		\begin{align}\label{4.4}
			\langle g, z \rangle_{W_0',W_0} = \langle m,\overline{\lambda}\rangle_{L^2(0,T;D(A)'),L^2(0,T;D(A))}.
		\end{align}
		Choose $g \in W_0'$ such that
		\begin{align}\label{4.5}
			\langle g,z \rangle_{W_0',W_0}= \int_0^T\langle g'_y(t,\overline{y}(t)), z(t) \rangle_{D(A)',D(A)}\,dt
		\end{align}
		for all $z \in W_0$.\\
		Now we shall prove that $\overline{\lambda} \in L^2(0,T;\mathbb{H}^3_\sigma)$. Let $\{w_k\}_{k\geq1}$ be the sequence of eigenfunctions of $A$ in $D(A)$. Then we can write $\overline{\lambda}(t)$ as 
		\begin{align*}
			\overline{\lambda}(t) = \sum_{k=0}^\infty d_k(t)w_k
		\end{align*}
		where $d_k(t)=(A\overline{\lambda}(t),Aw_k)$ are Fourier coefficients. Let $P_N$ be the projection map to from $D(A)$ to the space ${\rm span}\{w_1,\hdots,w_N\}$, i.e. $P_N\overline{\lambda}(t) = \sum_{k=0}^N d_k(t)w_k$. We can extend the operator $P_N$ to $D(A)'$ by defining the duality $\langle P_Ng,u\rangle_{D(A)',D(A)}=\langle g,P_Nu\rangle_{D(A)',D(A)}$. Denote $\overline{\lambda}_N=P_N\overline{\lambda}$ then we have the equation 
		\begin{equation} \label{estimate_adj_sol}
			-\partial_t(\overline{\lambda}_N+\alpha^2A\overline{\lambda}_N)+\nu A (\overline{\lambda}_N+ \alpha^2 A \overline{\lambda}_N) +P_N\widetilde{B}'(\overline{y},\overline{y}+\alpha^2A\overline{y})^*\overline{\lambda}_N = P_Ng'_y(t,\overline{y}).
		\end{equation}
		Notice that $w_k\in D(A^n)$ for all $n\geq1$ since we have $A^nw_k=\mu_k^n w_k$, where $\mu_n$ is the eigenvalue associated with $w_k$. Hence for $N\geq 0$, $\overline{\lambda}_N\in D(A^n)$.  We now estimate $\Vert \overline{\lambda}_N \Vert_{L^2(0,T;\mathbb{H}^3_\sigma)}$. Taking the inner product of \eqref{estimate_adj_sol} with $A\overline{\lambda}_N$ to obtain
		\begin{align}
			&-\frac{1}{2} \frac{d}{dt} \left(\Vert \overline{\lambda}_N \Vert^2 + \alpha^2 \vert A\overline{\lambda}_N \vert^2 \right) + \nu \left( \vert A\overline{\lambda}_N \vert ^2 + \alpha^2 \vert \nabla A \overline{\lambda}_N \vert^2 \right) + \widetilde{b} (A\overline{\lambda}_N, \overline{y} + \alpha^2 A \overline{y},  \overline{\lambda}_N) \nonumber \\
			&+\widetilde{b} (\overline{y}, A\overline{\lambda}_N+ \alpha^2 A^2 \overline{\lambda}_N, A \overline{\lambda}_N) = \langle g'_y(t,\overline{y}), A\overline{\lambda}_N\rangle_{V',V}. \label{innerproductAlambda}
		\end{align}
		Then since $ \overline{y} \in W(0,T;H^3,V) $ and $ \overline{\lambda}_N \in L^2(0,T;D(A)) $, we approach the nonlinear terms of \eqref{innerproductAlambda} in the following ways
		\begin{align*}
			\vert \langle g'_y(t,\overline{y}), A\overline{\lambda}_N\rangle_{V',V} \vert &\leq \| g'_y(t,\overline{y}) \|_{V'} \| A\overline{\lambda}_N \| \leq C\| g'_y(t,\overline{y}) \|^2_{V'} + \frac{\nu\alpha^2}{K} |\nabla A \overline{\lambda}_N |^2, \\
			\vert \widetilde{b}(A\overline{\lambda}_N, \overline{y} + \alpha^2 A \overline{y},  \overline{\lambda}_N) \vert &\leq \vert \widetilde{b}(A\overline{\lambda}_N,y, \overline{\lambda}_N) \vert + \alpha^2 \vert \widetilde{b}(A\overline{\lambda}_N, A\overline{y},\overline{\lambda}_N) \vert \\
			&\leq \vert b(A\overline{\lambda}_N, y, \overline{\lambda}_N) \vert +\vert b(\overline{\lambda}_N, y, A\overline{\lambda}_N) \vert \\
			&\hspace{10pt}+\alpha^2 \vert b(A\overline{\lambda}_N, Ay, \overline{\lambda}_N) \vert + \alpha^2 \vert b(A\overline{\lambda}_N, \overline{\lambda}_N, Ay) \vert \\
			&\leq C \vert \nabla A\overline{\lambda}_N \vert \vert A\overline{\lambda}_N \vert | A\overline{y} |+C\|\overline{\lambda}\||A\overline{\lambda}_N||Ay|\\
			&\leq C\vert A \overline{\lambda}_N \vert^2 + \dfrac{\nu\alpha^2}{K} \vert \nabla A\overline{\lambda}_N \vert^2 + \nu|A\overline{\lambda}_N|^2.
		\end{align*}
		To bound the term $ \widetilde{b} (\overline{y}, A\overline{\lambda}_N+ \alpha^2 A^2 \overline{\lambda}_N, A \overline{\lambda}_N) $, we start with the definition of the operator
		\begin{align*}
			\vert \widetilde{b} (\overline{y}, A\overline{\lambda}_N+ \alpha^2 A^2 \overline{\lambda}_N, A \overline{\lambda}_N) \vert &\leq \vert \widetilde{b}(\overline{y}, A\overline{\lambda}_N, \overline{\lambda}_N) \vert + \alpha^2 \widetilde{b}(y, A^2 \overline{\lambda}_N, \overline{\lambda}_N) \vert \\
			&\leq \vert b(\overline{y}, A\overline{\lambda}_N, \overline{\lambda}_N) \vert + \vert b(\overline{\lambda}_N, A\overline{\lambda}_N, \overline{y}) \vert \\
			&\hspace{10pt}+\alpha^2 \vert b(\overline{y}, A^2\overline{\lambda}_N, \overline{\lambda}_N) \vert + \alpha^2 \vert b(\overline{\lambda}_N, A^2 \overline{\lambda}_N, \overline{y}) \vert\\
			&\leq C \Vert \overline{y} \Vert_{D(A)} \vert \nabla A\overline{\lambda}_N \vert \Vert \overline{\lambda}_N \Vert \\
			&\hspace{10pt}+\alpha^2 \vert b(\overline{y}, \overline{\lambda}_N, A^2\overline{\lambda}_N) \vert + \alpha^2 \vert b(\overline{\lambda}_N,\overline{y}, A^2 \overline{\lambda}_N) \vert.
		\end{align*}
		Set $w_N=\Delta \overline{\lambda}_N$. For the term $b(\overline{y}, \overline{\lambda}_N, A^2\overline{\lambda}_N)$, using integral by parts, we continue as follows
		\begin{align*}
			b(\overline{y}, \overline{\lambda}_N, A^2\overline{\lambda}_N) &= \langle (\overline{y}\cdot\nabla)\overline{\lambda}_N,\Delta^2\overline{\lambda}_N\rangle_{V',V}\\
			&=\sum_{i,j=1}^3\int_\Omega \overline{y}_j \dfrac{\partial\overline{\lambda}_{Nj}}{\partial x_j}\Delta w_{Nj} \,dx\\
			&=\sum_{i,j=1}^3 \sum_{k=1}^3\int_\Omega \overline{y}_j \dfrac{\partial\overline{\lambda}_{Nj}}{\partial x_j}\dfrac{\partial^2w_{Nj}}{\partial x_k^2}\,dx\\
			&=-\sum_{i,j=1}^3 \sum_{k=1}^3\int_\Omega \dfrac{\partial}{\partial x_k}\left( \overline{y}_j \dfrac{\partial\overline{\lambda}_{Nj}}{\partial x_j}\right)\dfrac{\partial w_{Nj}}{\partial x_k}\,dx\\
			&=-\sum_{k=1}^3 \sum_{i,j=1}^3 \int_\Omega \left(\dfrac{\partial \overline{y}_j}{\partial x_k} \dfrac{\partial\overline{\lambda}_{Nj}}{\partial x_j} + \overline{y}_j\dfrac{\partial^2\overline{\lambda}_{Nj}}{\partial x_k\partial x_j}\right)\dfrac{\partial w_{Nj}}{\partial x_k} \,dx\\
			&=-\sum_{k=1}^3 \left(b(\nabla_k \overline{y}, \overline{
				\lambda}_N,\nabla_k w_N) + b( \overline{y}, \nabla_k \overline{
				\lambda}_N,\nabla_k w_N) \right)
		\end{align*}
		where $\nabla_k =\dfrac{\partial}{\partial x_k}$. Thus, \[\vert b(\overline{y}, \overline{\lambda}_N, A^2\overline{\lambda}_N) \vert \leq C | A\overline{y} | |A\overline{\lambda}| \vert \nabla A\overline{\lambda}_N \vert \leq C|A\overline{\lambda}_N|^2 +\dfrac{\nu\alpha^2}{K}|\nabla A\overline{\lambda}_N|^2.\]
		Estimate similarly for $b(\overline{\lambda}_N,\overline{y}, A^2 \overline{\lambda}_N)$, then choose $K=8$ and put all these bounds together yield
		\begin{equation} \label{bound_estimate_adjsol}
			-\frac{d}{dt} \left(\Vert \overline{\lambda}_N \Vert^2 + \alpha^2 \vert A\overline{\lambda}_N \vert^2 \right)  + \nu\alpha^2\vert \nabla A \overline{\lambda}_N \vert^2 \leq C \| g'_y(t,\overline{y}(t)) \|^2_{V'}.
		\end{equation}
		From \eqref{bound_estimate_adjsol}, integrating from $t$ to $T$ and passing through the limit as $N \to \infty$ infer that $\overline{\lambda}$ belongs to $L^2(0,T;\mathbb{H}^3_\sigma)$.\\
		Since $A$ is a self-adjoint operator, we have 
		\begin{align}
			\langle \nu A z,\overline{\lambda} \rangle_{L^2(0,T;D(A)'),L^2(0,T;D(A))} &=  \langle \nu A \overline{\lambda},z \rangle_{L^2(0,T;D(A))'),L^2(0,T;D(A))}, \label{4.6}\\
			\langle A^2 z,\overline{\lambda} \rangle_{L^2(0,T;D(A)'),L^2(0,T;D(A))} &=  \langle A^2 \overline{\lambda},z \rangle_{L^2(0,T;D(A))'),L^2(0,T;D(A))}, \label{4.6a}
		\end{align}
		and for the nonlinear term
		\begin{equation}
			\begin{split}
				\langle \widetilde{B}'(\overline{y}, & \overline{y} +\alpha^2 A\overline{y})z,\overline{\lambda}  \rangle_{L^2(0,T;D(A)'), L^2(0,T;D(A))} \\
				= \langle \widetilde{B}'(\overline{y},\overline{y} &+\alpha^2 A\overline{y})^*\overline{\lambda},z\rangle_{L^2(0,T;D(A)'),L^2(0,T;D(A))}.  \label{4.7}
			\end{split}
		\end{equation}
		
		We shall prove that $\overline{\lambda}_t$ exists as an element of $L^2(0,T;V)$. From \eqref{4.6}, \eqref{4.6a} and \eqref{4.7} we have
		\begin{equation}
			\begin{split}
				\int_0^T \langle z_t(t) + \alpha^2 Az_t(t),\overline{\lambda}(t)\rangle_{D(A)',D(A)}\,dt  \\
				= \int_0^T \langle g_y'(t,y(t)) - \nu A (\overline{\lambda}(t)+\alpha^2 A\overline{\lambda}(t)) - \widetilde{B}'(\overline{y},\overline{y} &+\alpha^2 A\overline{y})^*\overline{\lambda}(t),z(t) \rangle_{D(A)',D(A)}\,dt.\label{4.8}    
			\end{split}
		\end{equation}
		For $v \in D(A)$, set $z(t)=\varphi(t)v$ where $\varphi(t) \in C^{\infty}(0,T)$  and \[G(t)=g_y'(t,y(t)) - \nu A (\overline{\lambda}(t)+\alpha^2 A\overline{\lambda}(t)) - \widetilde{B}'(\overline{y},\overline{y} +\alpha^2 A\overline{y})^*\overline{\lambda}(t).\] Then we can write \eqref{4.8} as follows
		\begin{align} \label{4.9}
			\int_0^T \varphi'(t)\langle \overline{\lambda}(t) + \alpha^2 A\overline{\lambda}(t),v\rangle_{D(A)',D(A)}\,dt = \int_0^T \varphi(t)\langle G(t), v\rangle_{D(A)',D(A)} \,dt. 
		\end{align}
		Using Lemma 1.1 in \cite{temam1995navier}, we have
		\begin{align}\label{4.10}
			\int_0^T (\overline{\lambda}(t) + \alpha^2 A\overline{\lambda}(t))\varphi'(t) \,dt =\int_0^T G(t)\varphi(t) \,dt.
		\end{align}
		For convenience, define an operator $F:D(A) \to D(A)'$ by  $ F(u) = u + \alpha^2 Au$
		for all $u \in D(A)$. Then \eqref{4.10} can be rewritten as
		\begin{align*}
			\int_0^T F(\varphi'(t)\overline{\lambda}(t))\,dt = \int_0^T G(t)\varphi(t) \,dt. 
		\end{align*}
		From the definition of $F$, we have that $F$ is an isomorphism, and so is $F^{-1}$. Then from the properties of Bochner's integral (see e.g. \cite{yosida2012functional})
		\begin{align*}
			F\left(\int_0^T\overline{\lambda}(t)\varphi'(t)\,dt\right)  = \int_0^T G(t)\varphi(t) \,dt.
		\end{align*}
		Thus 
		\begin{align*}
			\int_0^T\overline{\lambda}(t)\varphi'(t)\,dt = F^{-1} \left(\int_0^T G(t)\varphi(t) \,dt\right) =\int_0^T F^{-1}(G(t))\varphi(t) \,dt.
		\end{align*}
		This holds for all $\varphi(t) \in C^{\infty}(0,T)$, therefore $\overline{\lambda}_t$ exists and $\overline{\lambda}_t = F^{-1}G$ in the weak sense. Moreover, since $G(t) \in L^2(0,T;D(A)')$, we obtain $\overline{\lambda}_t \in L^2(0,T;D(A))$. Thus $\overline{\lambda} \in W(0,T;H^3,V)$, and we have
		\begin{align*}
			\int_0^T \varphi'(t)\langle \overline{\lambda}(t) + \alpha^2 A\overline{\lambda}(t),v\rangle_{D(A)',D(A)}\,dt &= \int_0^T \langle \overline{\lambda}(t) + \alpha^2 A\overline{\lambda}(t),\varphi'(t)v\rangle_{D(A)',D(A)}\,dt\\
			=&-\int_0^T \langle \overline{\lambda}_t(t) + \alpha^2 A\overline{\lambda}_t(t),\varphi(t)v\rangle_{D(A)',D(A)}\,dt\\ =&-\int_0^T \varphi(t)\langle \overline{\lambda}_t(t) + \alpha^2 A\overline{\lambda}_t(t),v\rangle_{D(A)',D(A)}\,dt.
		\end{align*}
		This and \eqref{4.9} give 
		\begin{align*}
			\int_0^T \varphi(t)\langle \overline{\lambda}_t(t) + \alpha^2 A\overline{\lambda}_t(t),v\rangle_{D(A)',D(A)}\,dt = -\int_0^T \varphi(t)\langle G(t),v\rangle_{D(A)',D(A)} \,dt
		\end{align*}
		for all $\varphi \in C^{\infty}(0,T)$. Thus we obtain 
		\begin{align*}
			\langle \overline{\lambda}_t(t) + \alpha^2 A\overline{\lambda}_t(t),v\rangle_{D(A)',D(A)} = -\langle G(t),v\rangle_{D(A)',D(A)}
		\end{align*}
		for all $v\in V$. Taking $v=z(t)$ and integrating from $0$ to $T$ yield
		\begin{equation}
			\begin{split}
				\int_0^T  \langle & g'_y(t,\overline{y}(t)), z(t) \rangle_{D(A)',D(A)} \,dt = \int_0^T \langle -\overline{\lambda}_t(t) - \alpha^2 A \overline{\lambda}_t(t)  \\&+ \nu A (\overline{\lambda}(t)+\alpha^2 A\overline{\lambda}(t)) 
				+ \widetilde{B}'(\overline{y},\overline{y} +\alpha^2 A\overline{y})^*\overline{\lambda}(t),z(t)\rangle_{D(A)',D(A)} \,dt. \label{4.13}
			\end{split}
		\end{equation}
		This holds for all $z\in L^2(0,T;D(A))$, hence $\overline{\lambda}$ satisfies the first equation of \eqref{adj_eqt}. 
		
		By integration by parts, we have
		\begin{align}\label{4.15}
			{\langle z_t + \alpha^2 Az_t,\overline{\lambda} \rangle_{L^2(0,T;D(A)'),L^2(0,T;D(A))}} &= \int_0^T \langle z_t(t) + \alpha^2 Az_t(t),\overline{\lambda}(t) \rangle_{D(A)',D(A)} \,dt \nonumber \\ 
			&= \langle z(T) + \alpha^2 Az(T),\overline{\lambda}(T) \rangle_{D(A)',D(A)} \nonumber \\ - &\int_0^T  \langle z(t) + \alpha^2 Az(t),\overline{\lambda}_t(t) \rangle_{D(A)',D(A)}\,dt \nonumber\\
			&= \langle \overline{\lambda}(T) + \alpha^2 A\overline{\lambda}(T),z(T) \rangle_{D(A)',D(A)} \nonumber\\ - &\langle \overline{\lambda}_t + \alpha^2 A\overline{\lambda}_t,z \rangle_{L^2(0,T;D(A)'),L^2(0,T;D(A))}.
		\end{align}
		Combining \eqref{4.4}, \eqref{4.5}, \eqref{4.6}, \eqref{4.7} and \eqref{4.15}, we find 
		\begin{equation}\label{4.16}
			\begin{aligned}
				&\langle g'(\overline{y}), z\rangle_{L^2(0,T;D(A)'),L^2(0,T;D(A))}  = \langle \overline{\lambda}(T) + \alpha^2 A\overline{\lambda}(T),z(T) \rangle_{D(A)',D(A)} \\
				&+ \langle -\overline{\lambda}_t - \alpha^2 A \overline{\lambda}_t + \nu A (\overline{\lambda}+\alpha^2 A\overline{\lambda}) 
				+ \widetilde{B}'(\overline{y},\overline{y} +\alpha^2 A\overline{y})^*\overline{\lambda},z\rangle_{L^2(0,T;D(A)'),L^2(0,T;D(A))}.  
			\end{aligned}
		\end{equation}
		From \eqref{4.13} and \eqref{4.16}, we deduce that 
		$$\langle \overline{\lambda}(T) + \alpha^2 A\overline{\lambda}(T),z(T) \rangle_{D(A)',D(A)} = 0$$ for all $z \in W_0$, which implies $\overline{\lambda}$ satisfies the last equation of \eqref{adj_eqt}. Thus $\overline{\lambda}$ is the weak solution of \eqref{adj_eqt}. 
		For the uniqueness, we set $\delta \lambda=\lambda_1-\lambda_2$, where $\lambda_1$ and $\lambda_2$ are two distinct solutions of the equations \eqref{adj_eqt}. By applying analogous estimates to the equations satisfied by $\delta \lambda$, we deduce that $\delta \lambda=0$. The proof is complete.
	\end{proof}
	
	\begin{theorem}\label{theo4.1}
		Let $\overline{u}$ be a locally optimal control with associated state $\overline{y}$. Then there exists $\overline{\lambda}$, which is the weak solution of the adjoint equations \eqref{adj_eqt} such that
		\begin{equation}\label{4.2}
			\int_{0}^{T} \left( \overline{\lambda}(t)+\gamma \overline{u}(t)),m(t) \right) \,dt \ge 0,\quad  \forall m \in \mathcal{T}_{U_{a,b}}(\overline{u}).
		\end{equation}
		As a special case, the variational inequality
		\begin{equation} \label{4.3}
			\int_{0}^{T} \left( \overline{\lambda}(t)+\gamma \overline{u}(t),u(t)-\overline{u}(t) \right)\,dt \ge 0, \forall u \in U_{a,b},
		\end{equation}
		is satisfied.
	\end{theorem}
	
	Moving to our sparse problems $(\textbf{P}_k), k=1,2,3$. First, we give some properties of the functionals $j_k$ and the subdifferentials. We start with $j_1$.\\
	
	\textbf{Problem $(P_1)$}. \textcolor{red}{Since $j_1$ is convex and Lipchitz continuous, its subdifferential is \[\partial j_1(u)=\left\{ \zeta \in L^\infty(Q)^3: j_1(v)-j_1(u) \geq \iint_Q \zeta (v-u) \, dxdt \ \forall v\in L^1(Q)^3 \right \}.\]
		By taking $v=0$ and $v=2u$, we get that
		\[\iint_Q \zeta u \, dxdt = j_1(u) \ \forall \zeta \in \partial j_1(u).\]
		Hence,
		\[\iint_Q \zeta \, dxdt \leq j_1(v) \ \forall \zeta \in \partial j_1(u), \forall v\in L^1(Q)^3.\]
		Thus, if $\zeta \in \partial j_1(u)$ then }
	\begin{align} \label{j1_1}
		\begin{cases}
			\zeta_i(x,t)=1 &\ \text{if} \ u_i(x,t)>0,\\
			\zeta_i(x,t)=-1 &\ \text{if} \ u_i(x,t) <0,\\
			\zeta_i(x,t) \in [-1,1] &\ \text{if} \ u_i(x,t)=0,
		\end{cases}
		\hspace{13pt} i=1,2,3.
	\end{align}
	We also have that if $\zeta \in L^\infty(Q)^3$ and satisfies \eqref{j1_1} then $\zeta\in \partial j_1(u)$.\\
	Additionally, the directional derivatives of $j_1$ exist and are given by
	\begin{equation} \label{dir_deri_j1}
		\begin{aligned}
			j'_1(u)v&=\underset{\varepsilon\rightarrow0^+}{\lim} \frac{j_1(u+\varepsilon v)-j_1(u)}{\varepsilon} \\
			&=\sum_{i=1}^3 \left\{ \iint_{Q_{u,i}^+} v_i \, dxdt - \iint_{Q_{u,i}^-} v_i \, dxdt + \iint_{Q_{u,i}^0} \vert v_i \vert \, dxdt \right\},
		\end{aligned}
	\end{equation}
	where
	\begin{align*}
		Q_{u,i}^+ &= \left\{ (x,t)\in Q  \,| \, u_i(x,t)>0 \right\},\\
		Q_{u,i}^- &= \left\{ (x,t)\in Q \,| \, u_i(x,t)<0 \right\},\\
		Q_{u,i}^0 &= \left\{ (x,t)\in Q \,| \, u_i(x,t)=0 \right\}.
	\end{align*}
	Moreover, \eqref{j1_1} and \eqref{dir_deri_j1} yield
	\begin{equation*}
		\underset{\zeta  \in \partial j_1(u)}{\max} \iint_Q \zeta v\, dxdt = j'_1(u)v \leq \frac{j_1(u+\rho v)-j_1(u)}{\rho} \leq j_1(u+v)-j_1(u), \forall \rho \in (0,1).
	\end{equation*}
	We are now ready to derive the first-order necessary condition for problem $(\textbf{P}_1)$.
	
	\begin{theorem} \label{fnc_j1}
		Let $\overline{u}$ be a local solution to the problem $(\textbf{P}_1)$ with associated state $\overline{y} \in W(0,T;H^3,V)$. Then there exists \textcolor{red}{ $\overline{\zeta}\in \partial j_1(\overline{u})$ such that }
		\begin{equation} \label{var_ineq_j1}
			\iint_Q \left( \overline{\lambda} + \gamma \overline{u}(t)+ \kappa \overline{\zeta} \right)(u-\overline{u}) \,dxdt\geq 0 \ \forall u\in U_{a,b},
		\end{equation}
		where $\overline{\lambda}$ is the adjoint state, i.e. the unique solution of the adjoint equations \eqref{adj_eqt}.
	\end{theorem}
	\begin{proof}
		Since \textcolor{red}{$U_{a, b}$} is convex, we have that $\overline{u} + \rho (u - \overline{u}) \in \textcolor{red}{U_{a, b}}$ for every $\rho \in [0,1]$. From the local optimality of $\overline{u}$, we deduce $\mathcal{J}_1(\overline{u}+\rho(u - \overline{u})) - \mathcal{J}_1(\overline{u}) \geq 0$ for every $\rho$ sufficiently small. Hence,
		\begin{align*}
			0 &\leq \underset{\rho \rightarrow 0^+}{\lim} \frac{\mathcal{J}_1(\overline{u}+\rho (u - \overline{u})) - \mathcal{J}_1(\overline{u})}{\rho}\\
			&= \underset{\rho \rightarrow0^+}{\lim} \frac{\mathcal{J}(\overline{u}+\rho(u-\overline{u}))- \mathcal{J}(\overline{u})+\kappa j_1(\overline{u}+\rho(u-\overline{u})) - \kappa j_1(\overline{u}) }{\rho}\\
			&\leq \underset{\rho\rightarrow0^+}{\lim} \frac{\mathcal{J}(\overline{u}+\rho(u-\overline{u}))- \mathcal{J}(\overline{u})}{\rho} + \kappa (j_1(u) - j_1(\overline{u})) \\
			& =\iint_Q (\overline{\lambda} + \gamma \overline{u}(t))(u-\overline{u}) \, dxdt + \kappa (j_1(u)-j_1(\overline{u})).
		\end{align*}
		This implies that $\overline{u}$ is the solution of the optimization problem
		\begin{equation*}
			\underset{u\in U_{a, b}}{\min} I(u)= \iint_Q (\overline{\lambda}+\gamma \overline{u}(t))u \, dxdt + \kappa j_1(u).
		\end{equation*}
		Finally, since $I$ is convex, we can use the subdifferential calculus to obtain
		\begin{equation*}
			0 \in \partial I(\overline{u})=(\overline{\lambda} + \gamma \overline{u}(t)) + \kappa \partial j_1 (\overline{u}).
		\end{equation*}
		Then, by a classical result in convex analysis (see Proposition 27.8 in \cite{Bauschke2010}), there exists $\overline{\zeta}$ such that
		\begin{equation*}
			-(\overline{\lambda} + \gamma \overline{u}(t)+ \kappa \overline{\zeta}) \in \mathcal{N}_{U_{a,b}} (\overline{u}),
		\end{equation*}
		where $\mathcal{N}_{U_{a,b}}(\overline{u})$ is the normal cone of the set $U_{a,b}$ at $\overline{u}$. Thus we get \eqref{var_ineq_j1}.
	\end{proof}
	The following corollary proceeds deeper into the sparse properties of the optimal control $\overline{u}$. The interested readers can find its proof in \cite{CHW2012}.
	\begin{corollary}
		Let $\overline{u}, \overline{y},\overline{\zeta}$ be defined in Theorem \ref{fnc_j1}. Then the following relations hold for $i=1,2,3$
		\begin{align}
			\overline{u}_i(x,t) &= \text{Proj}_{[a_i,b_i]}\left\{ -\frac{1}{\gamma} \left( \overline{\lambda}_i(x,t) + \kappa \overline{\zeta}_i(x,t) \right)\right\} \label{j1_2},\\
			\overline{u}_i(x,t) &=0 \Leftrightarrow \left| \overline{\lambda}_i(x,t) \right| \leq \kappa \label{j1_3}, \\
			\overline{\zeta}_i(x,t) &= \text{Proj}_{[-1;1]} \left( -\frac{1}{\kappa} \overline{\lambda}_i(x,t) \right). \label{j1_4}
		\end{align}
		Moreover, it follows from \eqref{j1_2} and \eqref{j1_4} that $\overline{u},\overline{\zeta} \in L^2(0,T;\mathbb{H}^3_\sigma)$
		and $\overline{\zeta}$ is unique for any fixed local mininum $u$.
	\end{corollary}
	\textcolor{red}{
		\begin{remark}
			{\rm From the state equation \eqref{VCHE_weaksolution}, we obtain an $L^\infty(Q)^3 $-\,estimate for $\overline{y}$ depending on $a$ and $b$, but independent of $\overline{u}$. Then by using the corresponding adjoint equations \eqref{adj_eqt}, we deduce the bound $\Vert \overline{\lambda} \Vert_{L^\infty(Q)^3} \leq M$ with the constant $M$ independent of $\overline{u}$. Hence from \eqref{j1_3}, we conclude that $\overline{u} \equiv 0 $ whenever $\kappa>M$. Therefore, we may influence the size of an optimal control's support by adjusting $\kappa$ in the interval $[0,M]$. \\
				Another notable sparsity case occurs when $\nu=0$. Using arguments similar to those in \cite{AGS2025}, we confirm that if the set of points $(x,t)\in Q$ satisfying $\left| \overline{\lambda}(x,t) \right| = \kappa$ has a zero Lebesgue measure, then $\overline{u}(x,t)\in \{\alpha, 0,\beta\}$ for almost all $(x,t)\in Q$. The optimal control then has a bang-bang-bang structure.
			}
	\end{remark} }
	
	Next, we consider the time sparse \textbf{Problem $(P_2)$}.\\
	Remind that $\displaystyle j_2(u)=\Vert u \Vert_{L^2(0,T;\mathbb{L}^1(\Omega))}=\left[ \int_0^T \Vert u(t) \Vert_{\mathbb{L}^1(\Omega)}^2 dt \right] ^{1/2}$, its subdifferential and directional derivatives are stated as following
	\begin{proposition}
		For $u\neq 0$, \begin{equation} \label{j2_1}
			\partial j_2(u)=\left\{ \zeta\in L^2(0,T;\mathbb{L}^\infty(\Omega)) \, \left| \, \zeta(x,t) \in \text{Sign} (u(x,t)) \frac{\Vert u(t) \Vert_{\mathbb{L}^1(\Omega)}}{\Vert u \Vert_{L^2(0,T;\mathbb{L}^1(\Omega))}} \ \text{a.e in} \ Q \right. \right\}.
		\end{equation}
		In case $u=0$, we have $\displaystyle \partial j_2(u)=\left\{\zeta \in L^2(0,T;\mathbb{L}^\infty(\Omega)): \Vert \zeta \Vert_{L^2(0,T;\mathbb{L}^\infty(\Omega))} \leq 1 \right\}$. \\
		Moreover, for $\zeta\in \partial j_2(u)$ we have
		\begin{equation} \label{j2_2}
			\begin{aligned}
				\begin{cases}
					\text{supp} \ u^+(t) \subset \left\{ x \in \Omega: \zeta(x,t)=+\Vert \zeta (t) \Vert_{\mathbb{L}^\infty (\Omega)} \right\} \\
					\text{supp} \ u^-(t) \subset \left\{ x\in \Omega: \zeta(x,t) = -\Vert \zeta(t) \Vert_{\mathbb{L}^\infty(\Omega)} \right \}
				\end{cases}
			\end{aligned}
			\text{a.e in} \ (0,T).
		\end{equation}
	\end{proposition}
	
	\begin{proposition}
		For every $u,v\in L^2(0,T;\mathbb{L}^1(\Omega))$ we have either
		\begin{equation} \label{j2_3}
			j'_2(u)v= \frac{1}{\Vert u \Vert_{L^2(0,T;\mathbb{L}^1(\Omega))}} \int_0^T j'_\Omega(u(t))v(t) \Vert u(t) \Vert_{\mathbb{L}^1(\Omega)} \,dt
		\end{equation}
		in case  $u\neq 0$ or $j_2'(u)v=j_2(v)$ in case $u=0$. \\
		Here, $\displaystyle j'_\Omega(u(t))v(t)=\int_\Omega sign (u(x,t)) v(x,t) \, dx.$
	\end{proposition}
	From the previous propositions and Theorem \ref{fnc_j1}, we deduce the following corollary.
	\begin{corollary}
		Let $\overline{u}, \overline{\lambda}, \overline{\zeta}$ be as in Theorem \ref{fnc_j1}. The following relations hold for almost all $(x,t)\in Q$ and $i=1,2,3$
		\begin{align}
			\overline{u}_i(x,t)&=\text{Proj}_{[a_i,b_i]} \left( -\frac{1}{\gamma} \left[ \overline{\lambda}_i(x,t) + \kappa \overline{\zeta}_i(x,t) \right] \right), \label{j2_4} \\
			\overline{u}_i(x,t) &= 0 \Leftrightarrow \vert \overline{\lambda}_i(x,t) \vert \leq \kappa \overline{\sigma}_i(t), \label{j2_5} \\
			\overline{\zeta}_i(x,t) &= \text{Proj}_{[-\overline{\sigma}_i(t), +\overline{\sigma}_i(t)]} \left( -\frac{1}{\kappa} \overline{\lambda}_i(x,t) \right), \label{j2_6}
		\end{align}
		where
		\begin{align*}
			\overline{\sigma}_i(t)=\begin{cases}
				\frac{1}{\Vert \overline{u}_i \Vert_{L^2(0,T;L^1(\Omega))}} \Vert \overline{u}_i(t) \Vert_{L^1(\Omega)} \ &\text{if} \ \overline{u}_i\neq 0 \\
				1 \ &\text{if} \ \overline{u}_i = 0.
			\end{cases}
		\end{align*}
		Moreover, $\overline{u}$ and $\overline{\zeta}\in L^2(0,T;\mathbb{H}^3_\sigma)$ holds and $\overline{\zeta}$ is unique. 
	\end{corollary}
	
	Analogously to the problem $(\textbf{P}_1)$, we can prove that there exists $M>0$ independent of $\kappa$ such that $\overline{u} \equiv 0$ if $\kappa>M$. Hence, we can monitor the parameter $\kappa$ in the interval $[0,M]$ to get a suitable support for $\overline{u}$.
	
	Finally, we consider the space sparse \textbf{Problem $(P_3)$}.\\
	Since $\displaystyle j_3(u)=\Vert u\Vert_{L^1(\Omega;L^2(0,T)^3)}=\int_\Omega \Vert u(x) \Vert_{L^2(0,T)^3} \,dx$, its subdifferential and directional derivatives are given in the following propositions.
	\begin{proposition}
		The following statements hold \\
		1. $\zeta \in \partial j_3(u)$ is equivalent to $\zeta\in L^\infty(\Omega;L^2(0,T)^3)$ and
		\begin{equation} \label{j3_1}
			\displaystyle
			\begin{aligned}
				\begin{cases}
					\Vert \zeta(x) \Vert_{L^2(0,T)^3} \leq 1 &\ \text{for a.e} \ x\in Q \ \text{if} \ u=0\\
					\zeta(x,t) = \frac{u(x,t)}{\Vert u(x) \Vert_{L^2(0,T)^3}} &\ \text{for a.e} \ x\in Q \ \text{if} \ u\ne0.
				\end{cases}
			\end{aligned}
		\end{equation}
		2. For every $u,v \in L^1(\Omega,L^2(0,T)^3)$
		\begin{equation} \label{j3_2}
			j_3'(u)v= \int_{\Omega_u^0} \Vert v(x) \Vert _{L^2(0,T)^3} \, dx + \int_{\Omega_u} \frac{1}{\Vert u(x) \Vert_{L^2(0,T)^3}} \int_0^T uv \, dtdx,
		\end{equation}
		where
		\[\Omega_u=\left\{ x \in \Omega: \Vert u(x) \Vert_{L^2(0,T)^3} \neq 0 \right\} \ \text{and} \ \Omega_u^0 = \Omega \setminus \Omega_u.\]
	\end{proposition}
	Similar to problem $(\textbf{P}_1)$, we get the following corollary.
	\begin{corollary}
		Let $\overline{u}, \overline{\lambda}, \overline{\zeta}$ be as in Theorem \ref{fnc_j1}, then the following relations hold for almost all $(x,t)\in Q$ and $i=1,2,3$
		\begin{align}
			\overline{u}_i(x,t) &= \text{Proj}_{[a_i, b_i]} \left( -\frac{1}{\gamma} \left[ \overline{\lambda}_i(x,t) + \kappa \overline{\zeta}_i(x,t) \right] \right), \label{j3_3}\\
			\Vert \overline{u}_i(x,t) \Vert_{L^2(0,T)} &= 0 \Leftrightarrow \Vert \overline{\lambda}_i(x,t) \Vert_{L^2(0,T)} \leq \kappa, \label{j3_4}\\
			\overline{\zeta}_i(x,t) &= \text{Proj}_{[-\overline{\sigma}_i(t), +\overline{\sigma}_i(t)]} \left( -\frac{1}{\kappa} \overline{\lambda}_i(x,t) \right),\label{j3_5}
		\end{align}
		where
		\begin{align*}
			\overline{\sigma}_i(t)=\begin{cases}
				\frac{u_i(x,t)}{\Vert u_i(x) \Vert_{L^2(0,T)}} \ &\text{if} \ \overline{u}\neq 0 \\
				1 \ &\text{if} \ \overline{u} = 0.
			\end{cases}
		\end{align*}
		Moreover, $\overline{u}$ and $\overline{\zeta}\in L^2(0,T;\mathbb{H}^3_\sigma)$ holds and $\overline{\zeta}$ is unique. 
	\end{corollary}
	
	\begin{remark}
		\rm{
			The three problems differ in the sparsity properties of their corresponding local solutions $\overline{u}$. From \eqref{j1_2} and \eqref{j1_3}, we observe that the local solutions of $(\textbf{P}_1)$ are sparse in both space and time. Meanwhile, \eqref{j2_4} and \eqref{j2_5} indicate that the sparsity region of the solutions to $(\textbf{P}_2)$ may vary over time. On the other hand, the sparsity region of $(\textbf{P}_3)$ remains constant throughout time, and its solutions are therefore sparse only in space.
		}
	\end{remark}
	\begin{remark}
		\rm{We can consider a more abstract form of the term $\displaystyle \dfrac{\gamma}{2} \iint_Q \vert u(x,t)\vert^2 dxdt$ by using a function $\displaystyle \int_0^T h(u(t))dt$ with $h: \mathbb{L}^2(\Omega) \rightarrow \mathbb{R}$ satisfying some conditions as in \cite{AS2025}. In that case, the first-order necessary conditions subject to $J_k(y,u)$ can still be established, although the corresponding sparse structure becomes more difficult to describe explicitly.
		}
	\end{remark}
	\section{Second-order optimality conditions}
	
	Following the approach in Section 4, the second-order of $g$ with respect to $y$ in direction $(w_1,w_2)\in V\times V$ is
	\begin{align*}
		\langle g_{yy}''(t,y)w_1,w_2 \rangle_{V',V}=\int_{\Omega}g_{yy}''(t,y)[w_1(x),w_2(x)]\,dx.
	\end{align*}
	
	Now we move onto the sparse problems $(\textbf{P}_k), k=1,2,3$.
	The second-order conditions will be established based on the critical cone $\mathcal{C}_{\overline{u}}$, which refers to the set of all functions $v\in L^2(0,T;\mathbb{L}^2(\Omega))$ such that 
	\begin{align}
		\mathcal{J}'(\overline{u})v &+ \kappa j_k'(\overline{u})v=0, \nonumber\\
		v_i(x,t) &\geq0 \ \text{if} \ \overline{u}_i(x,t)=a_i, i=1,2,3, \label{NUad_1} \\
		v_i(x,t) &\leq 0 \ \text{if} \ \overline{u}_i(x,t)= b_i, i =1,2,3. \label{NUad_2}
	\end{align}
	The set $\mathcal{C}_{\overline{u}}$ is a closed, convex cone in $L^2(0,T;\mathbb{L}^2(\Omega))$.
	
	Before studying the second-order necessary and sufficient conditions for problems $(\textbf{P}_k)$, we prove the following auxiliary lemma. Observe that we only state the main results corresponding to $j_1$. For $j_2$ and $j_3$, analogous conclusions can be derived by adapting the arguments in \cite{CHW2017}.
	\begin{lemma} \label{5.1}
		Assume that $\overline{u} \in \textcolor{red}{U_{a,b}} $ satisfies \eqref{var_ineq_j1} for some $\overline{\zeta} \in \partial j_1(\overline{u})$ and $v\in L^2(0,T;\mathbb{L}^2(\Omega))$ satisfies \eqref{NUad_1} - \eqref{NUad_2}. Then
		\begin{equation*}
			\mathcal{J}'(\overline{u})v + \kappa j_1'(\overline{u})v \geq \mathcal{J}'(\overline{u})v + \kappa \iint_Q \overline{\zeta}v \, dxdt  \geq 0.
		\end{equation*}
		Moreover, if $v \in \mathcal{C}_{\overline{u}}$ then 
		\begin{equation}
			\mathcal{J}'(\overline{u})v + \kappa \iint_Q \overline{\zeta}v \, dxdt = 0 \ \text{and} \ j_1'(\overline{u})v=\iint_Q \overline{\zeta}v \, dxdt,\label{lem5.1}
		\end{equation}
		or equivalently, $\displaystyle \vert v_i(x,t)\vert = \zeta_i(x,t) v_i(x,t)$ for a.e $(x,t) \in Q_{\bar{u}}^0$ and every $i=1,2,3$.
	\end{lemma}
	\begin{proof}
		The first inequality follows from the fact that $\displaystyle \underset{\overline{\zeta}\in \partial j_1(\overline{u})}{\max} \iint_Q \overline{\zeta}v \,dxdt =j_1'(\overline{u})v$.
		To prove the remaining affirmations, we construct a sequence $\{ v_k \}$ to approximate $v$ as follows
		\begin{equation} \label{5.1.1}
			\begin{aligned} v_{k,i}(x,t)=
				\begin{cases}
					0 \ \text{if} \ a_i < \overline{u}_i(x,t) < a_i + \frac{1}{k} \ \text{or} \ b_i -\frac{1}{k} < \overline{u}_i(x,t) < b_i,\\
					\text{Proj}_{[-k,k]} \left(v_i(x,t)\right) \ \text{otherwise}, k\in \mathbb{N}.
				\end{cases}
			\end{aligned}
		\end{equation}
		Set $\rho_k=\frac{1}{k^2}$, then $\overline{u}+\rho v_k \in U_{a,b} \ \forall \, 0<\rho<\rho_k$.
		From definition \eqref{5.1.1}, we get that $\displaystyle \vert v_{k,i}(x,t) \vert \leq \vert v_i(x,t) \vert \ \forall i$ and $\displaystyle v_{k,i}(x,t) \rightarrow v(x,t)$ for a.e $x\in \Omega$, thus $v_k \rightarrow v$ in $L^2(Q)^3$.
		Consequently,
		\begin{align*}
			\rho & \left[ J'(\overline{u})v_{k,i} + \kappa \iint_Q \overline{\zeta}v_{k,i}(x,t) \, dxdt \right] \\ 
			&= \iint_Q \left( \overline{\lambda} + \gamma \overline{u} + \kappa \overline{\zeta} \right) \left[ \left( \overline{u} + \rho v_{k,i} \right) - \overline{u} \right] \, dxdt \geq 0.
		\end{align*}
		Passing to the limit as $k \rightarrow \infty$ yields $\displaystyle J'(\overline{u})v + \kappa \iint_Q \overline{\zeta} v \, dxdt \geq 0$.
	\end{proof}
	
	Henceforth, for convenience, we define a function $\overline{\eta}(x,t)=\overline{\lambda}(x,t) + \gamma \overline{u}(x,t) + \kappa \overline{\zeta}(x,t).$ Applying the above-mentioned results, we get the following properties of $\overline{\eta}(x,t)$
	\begin{equation} \label{eta_1}
		\begin{aligned}
			\begin{cases}
				\overline{u}_i(x,t)=a_i &\Rightarrow  \overline{\eta}(x,t) \geq 0, \\
				\overline{u}_i(x,t) =b_i &\Rightarrow \overline{\eta}(x,t) \leq 0, \hspace{20pt} 1\leq i \leq 3, \\
				a_i  < \overline{u}_i(x,t) < b_i &\Rightarrow \overline{\eta}(x,t) =0,
			\end{cases}
		\end{aligned}
	\end{equation}
	and 
	\begin{equation} \label{eta_2}
		\begin{aligned}
			\begin{cases}
				\overline{\eta}_i(x,t) > 0 &\Rightarrow \overline{u}_i(x,t) = a_i, \\
				\overline{\eta}_i(x,t)<0 &\Rightarrow \overline{u}_i(x,t) =b_i,
			\end{cases} 
		\end{aligned} \hspace{20pt} 1\leq i \leq 3 \textcolor{red}{.}
	\end{equation}
	It follows from the variational inequality that 
	\begin{equation*}
		\iint_Q \overline{\eta}v \, dxdt = \mathcal{J}'(\overline{u})v + \kappa \iint_Q \overline{\zeta}v \, dxdt = 0 \ \forall v\in \mathcal{C}_{\overline{u}}.
	\end{equation*}
	Together with the definition of $\mathcal{C}_{\overline{u}}$ we get that
	\begin{equation} \label{eta_3}
		\overline{\eta}_i(x,t)v_i(x,t) = 0 \ \text{for a.e} \ (x,t)\in Q, \forall i = \overline{1,3}, \forall v\in \mathcal{C}_{\overline{u}}.
	\end{equation}
	Due to the uncertainty in the existence of second directional derivatives of the functions $j_i, i =1,2,3$, the following theorems give the second-order optimality conditions for $\mathcal{J}$ with respect to the critical cone $\mathcal{C}_{\overline{u}}$.
	\begin{theorem} \label{theo_snc}
		If $\overline{u}$ \textcolor{red}{is} a local minimum of $(\textbf{P}_k)$ then
		\begin{equation} \label{snc}
			\mathcal{J}''(\overline{u})v^2 \geq 0 \ \forall v\in \mathcal{C}_{\overline{u}}.
		\end{equation}
	\end{theorem}
	\begin{proof}
		Given $v \in \mathcal{C}_{\overline{u}}$ we define $v_k=(v_{k,1}, v_{k,2}, v_{k,3})$ for $k\in \mathbb{Z}^+$ by
		\begin{equation*}
			\begin{aligned}
				v_{k,i}(x,t)=\begin{cases}
					0 \ \text{if} \ a_i < \overline{u}_i(x,t) < a_i + \frac{1}{k} \ \text{or} \ b_i - \frac{1}{k} < \overline{u}_i(x,t) < b_i \\ \hspace{130pt} \text{or} \ 0< \vert \overline{u}_i(x,t)\vert < \frac{1}{k},\\
					\text{Proj}_{[-k,k]} (v_i(x,t)) \ \text{otherwise}, \hspace{10pt} i=1,2,3.
				\end{cases}
			\end{aligned}
		\end{equation*}
		Then using the definition of $\overline{\eta}(x,t)$, we can check that
		\begin{align*}
			\overline{u}+ \rho v_k \in U_{a,b} \ &\text{for every} \ \rho \in \left( 0, \frac{1}{k^2} \right), \\
			v_k \rightarrow v \ &\text{in} \ L^2(Q)^3, \\
			\overline{\eta}_i(x,t) v_{k,i}(x,t) = 0 \ &\text{for a.e} \ (x,t)\in Q, \forall i =1,2,3.
		\end{align*}
		Now since $\overline{u}$ is a local minimum, we consider the difference
		\begin{equation*}
			\mathcal{J}_1 (\overline{u}+\rho v_k) - \mathcal{J}_1(\overline{u}) = \mathcal{J}(\overline{u}+\rho v_k) - \mathcal{J}(\overline{u}) + \kappa(j_1(\overline{u}+\rho v_k) - j_1(\overline{u})) \geq 0,
		\end{equation*}
		or equivalently,
		\begin{equation*}
			\rho \mathcal{J}'(\overline{u})v_k + \frac{\rho^2}{2} \mathcal{J}''(\overline{u})v_k^2 + o(\rho^2) +\kappa(j_1(\overline{u}+\rho v_k) - j_1(\overline{u})) \geq 0.
		\end{equation*}
		Let $\overline{\zeta}$ be the unique element of $\partial j(\overline{u})$ as in Theorem \ref{fnc_j1}, then
		\begin{equation} \label{5.2.1}
			j_1(\overline{u}+\rho v_k) - j_1(\overline{u}) = \rho \sum_{i=1}^3 \iint_Q \overline{\zeta}(x,t) v_i(x,t) \, dxdt \hspace{10pt} \forall \rho \in \left( 0, \frac{1}{k^2} \right).
		\end{equation}
		Substituting \eqref{5.2.1} into the above inequality yields
		\begin{equation*}
			\rho \sum_{i=1}^3 \iint_Q \overline{\eta}_i(x,t) v_{k,i}(x,t) \, dxdt + \frac{\rho^2}{2} \mathcal{J}''(\overline{u})v_k^2 + o(\rho^2) \geq 0.
		\end{equation*}
		Since the first term of the left-hand side is equal to $0$ by the property of $\overline{\eta}(x,t)$, we get
		\begin{equation*}
			\frac{\rho^2}{2} \mathcal{J}''(\overline{u})v_k^2 + o(\rho^2) \geq 0.
		\end{equation*}
		Now, fixing $k$, dividing both side by $\rho^2$ and passing to the limit imply that $\mathcal{J}''(\overline{u})v_k^2 \geq 0$. Then, passing to the limit again when $k\rightarrow \infty $ gives rise to the desired statement \eqref{snc}. The proof is complete.
	\end{proof}
	\begin{theorem} \label{theo_ssc}
		Let $\overline{u} \in U_{a, b}$ and $\overline{\zeta}\in \partial j_1(\overline{u})$ such that \eqref{var_ineq_j1} holds. Then the following statements are equivalent:
		\begin{enumerate}
			\item $\mathcal{J}''(\overline{u})v^2 > 0$ for every $v\in \mathcal{C}_{\overline{u}} \setminus \{0 \}. $\\
			\item There exists $\tau>0$ and $\delta>0$ such that
			\begin{equation*}
				\mathcal{J}''(\overline{u})v^2 \geq \delta \Vert v \Vert^2 _{L^2(0,T;\mathbb{L}^2(\Omega))}
			\end{equation*}
			for every $v \in \mathcal{C}^\tau_{\overline{u}}$. Here $\mathcal{C}_{\overline{u}}^\tau$ consists of all $v \in L^2(0,T;\mathbb{L}^2(\Omega))$ that satisfies \eqref{NUad_1}, \eqref{NUad_2} and the following inequality
			\begin{equation*}
				\mathcal{J}'(\overline{u})v + \kappa j_1'(\overline{u})v \leq \tau \Vert v \Vert_{L^2(0,T;\mathbb{L}^2(\Omega))}.
			\end{equation*}
		\end{enumerate} 
	\end{theorem}
	
	\begin{proof}
		Given that $\mathcal{C}_{\overline{u}} \subset \mathcal{C}_{\overline{u}}^\tau$, the first statement appears to be the direct consequence of the second one. 
		Let us prove the converse implication by contradiction. Assume that the first condition holds, whereas the second one does not. Then, there exists a sequence $\{ v_k \} $ such that
		\begin{equation*}
			v_k \in \mathcal{C}_{\overline{u}}^{1/k} \ \text{and} \ \mathcal{J}''(\overline{u}
			)v_k^2 < \frac{1}{k} \Vert v_k \Vert_{L^2(0,T;\mathbb{L}^2(\Omega))}^2.
		\end{equation*}
		Since $v_k \ne 0$ for every $k$, we can set $w_k = \frac{v_k}{\Vert v_k \Vert_{L^2(0,T;\mathbb{L}^2(\Omega))}}$ to obtain
		\begin{equation} \label{5.3.1}
			w_k \in \mathcal{C}_{\overline{u}}^{1/k}, \Vert w_k \Vert_{L^2(0,T;\mathbb{L}^2(\Omega))}=1 \ \text{and} \ \mathcal{J}''(\overline{u})w_k^2 < \frac{1}{k}.
		\end{equation}
		From the boundedness, we can extract from $\{ w_k\}$ a subsequence, denoted again by $\{ w_k\}$, converging weakly to some $v \in L^2(Q)^3$. We will show that $v$ satisfies \eqref{NUad_1}, \eqref{NUad_2}, and belongs to the critical cone $\mathcal{C}_{\overline{u}}$, yet is inconsistent with the second-order sufficient conditions. \\
		Indeed, suppose that the sign condition \eqref{NUad_1} is not true. Then for some $i$ and $k$, one can find a subset $A$ of the set $\{(x,t) \in Q: \overline{u}_i(x,t) = a_i \} $ with positive measure, such that 
		\begin{equation} \label{5.3.2}
			v_i(x,t) <0 \ \forall  (x,t) \in A
		\end{equation}
		We define $\xi \in L^2(0,T;L^2(\Omega))$ by
		\begin{equation*} \xi(x,t)=
			\begin{aligned}
				\begin{cases}
					1 \ &\text{if} \ (x,t) \in A,\\
					0 \ &\text{if} \ (x,t) \in Q \setminus A.
				\end{cases}
			\end{aligned}
		\end{equation*}
		Since $w_k \rightharpoonup v$ in $L^2(0,T;\mathbb{L}^2(\Omega))$, we obtain that $w_{k,i} \rightharpoonup v_i$ in $L^2(0,T;L^2(\Omega))$, thus
		\begin{equation*}
			\int_A w_{k,i} (x,t) \, dxdt \rightarrow \int_A v_i(x,t) \, dxdt.
		\end{equation*}
		However, the above relation fails to hold because $w_k \in C_{\overline{u}}^{1/k}$. Consequently, $v$ satisfies the first sign condition \eqref{NUad_1}, while the remaining condition \eqref{NUad_2} follows by similar arguments. \\
		To prove that $v$ belongs to the critical cone $\mathcal{C}_{\overline{u}}$, taking into account that the mapping $w \mapsto \mathcal{J}'(\overline{u})w + \kappa j_1'(\overline{u})w\in \mathbb{R}, w\in L^2(0,T;\mathbb{L}^2(\Omega))$ is weakly lower semicontinuous.
		Hence, 
		\begin{equation*}
			\mathcal{J}'(\overline{u})v + \kappa j_1'(\overline{u})v \leq \underset{k\rightarrow \infty}{\liminf} \left( \mathcal{J}'(\overline{u})w_k + \kappa j_1'(\overline{u})w_k \right) \leq \underset{k\rightarrow \infty}{\liminf} \frac{1}{k} = 0.
		\end{equation*}
		Combining with the first-order necessary condition leads to $\mathcal{J}'(\overline{u})v + \kappa j_1'(\overline{u})v =0$, implying that $v \in \mathcal{C}_{\overline{u}}$. Now we shall prove that $v=0$, which finally gives rise to the contradiction.
		Remind that
		\begin{equation*}
			\mathcal{J}''(\overline{u})v^2 = \int_0^T \left[ g''_{yy}(t, \overline{y}(t))\overline{z}^2_v + g'_y(t,\overline{y}(t))\overline{m}_v \right] \,dt + \gamma \iint_Q \vert v(x,t) \vert^2 \, dxdt,
		\end{equation*}
		where $\overline{y}=S(\overline{u}), \overline{z}_v=S'(\overline{u})v$, and $\overline{m}_v=S''(\overline{u})v^2$. \\ From the linearized equations, we get the boundedness of the sequence $\{ \overline{z}_{w_k} \}$ in the space $W(0,T;H^3, V)$. Hence, we can extract from $\{ \overline{z}_{w_k} \}$ a subsequence that converges weakly in $W(0,T;H^3, V)$, which is continuously embedded in $C(0,T;V)$ by Aubin-Lions theorem. Additionally, the function $g$ is strictly convex, thus $g''_{yy}(t, \overline{y}(t))$ is a positive semidefinite bilinear form, meaning that $g''_{yy}(t, \overline{y}(t))$ is convex. Then, $g''_{yy}(t, \overline{y}(t))$ is weakly lower semicontinuous and 
		\begin{equation} \label{5.3.3}
			g''_{yy}(t, \overline{y}(t))\overline{z}_v^2(t) \leq \underset{k\rightarrow \infty}{\liminf} g''_{yy}(t, \overline{y})\overline{z}^2_{w_k}
		\end{equation}
		for all $t \in [0,T]$. Similarly, we have 
		\begin{equation*}
			\begin{aligned}
				\overline{m}_{w_k} = S''(\overline{u})w_k^2 = -2S'(\overline{u})\left[ \widetilde{B}(w_k, w_k + \alpha^2 A w_k) \right] \\ \rightharpoonup -2S'(\overline{u})\left[ \widetilde{B}(v, v + \alpha^2 A v) \right] = S''(\overline{u})v^2 =\overline{m}_v \in 
				W(0,T;H^3,V).  
			\end{aligned}
		\end{equation*}
		Thus we obtain
		\begin{equation} \label{5.3.4}
			g'(\overline{y}(t))\overline{m}_v(t) = \underset{k\rightarrow \infty}{\lim}  g'(\overline{y}(t))\overline{m}_{w_k}(t).
		\end{equation}
		Combing \eqref{5.3.3}, \eqref{5.3.4} and Lemma \ref{lem3.1} yields
		\begin{equation*}
			\mathcal{J}(\overline{u})v^2 \leq \underset{k \rightarrow \infty }{\liminf} \mathcal{J}''(\overline{u})w_k^2 \leq \underset{k \rightarrow \infty}{\limsup} \mathcal{J}''(\overline{u})w_k^2 \leq \underset{k \rightarrow \infty}{\liminf} \frac{1}{k} = 0.
		\end{equation*}
		From the first condition, we get that $v=0$ and $\displaystyle \underset{k\rightarrow \infty }{\lim} \mathcal{J}''(\overline{u})w_k^2 = 0$.
		Since $v=0$ we have $\overline{z}_v=0$ and then $\overline{z}_{w_k} \rightharpoonup 0 $ in $W(0, T;H^3, V)$. Hence, $\mathcal{J}''(\overline{u})w_k^2 \rightarrow \gamma$, which leads to the contradiction. 
	\end{proof}
	\begin{theorem}
		Assume that $(\overline{u}, \overline{\zeta}) \in U_{a, b} \times \partial j_1(\overline{u})$ satisfies \eqref{var_ineq_j1}. Furthermore, let the following second-order sufficient condition be fulfilled
		\begin{equation} \label{ssc}
			\mathcal{J}''(\overline{u})v^2 >0 \ \forall v \in \mathcal{C}_{\overline{u}} \setminus \{0 \},
		\end{equation}
		then there exists $\varepsilon>0$ and $\delta >0$ such that
		\begin{equation} \label{grow_ineq}
			\mathcal{J}_1 (\overline{u})+ \frac{\delta}{2} \Vert u - \overline{u} \Vert^2_{L^2(0,T;\mathbb{L}^2(\Omega))} \leq \mathcal{J}_1 (u) \ \forall u \in B_\varepsilon(\overline{u}) \cap U_{a,b},
		\end{equation}
		where $\displaystyle B_\varepsilon(\overline{u}) = \{ u\in L^2(0,T;\mathbb{L}^2(\Omega)) : \Vert u - \overline{u} \Vert_{L^2(0,T;\mathbb{L}^2(\Omega))} \leq \varepsilon \}.$
		Moreover, since $\overline{u}$ satisfying \eqref{grow_ineq} is also a local solution of the following problem
		\begin{equation*}
			\underset{u \in U_{a, b}}{\min} F(u) := \mathcal{J}_1(u) - \frac{\delta}{2} \Vert u - \overline{u} \Vert^2_{L^2(0,T;\mathbb{L}^2(\Omega))},
		\end{equation*}
		we consequently get the following inequality
		\begin{equation*}
			\mathcal{J}''(\overline{u})v^2 \geq \delta\Vert v \Vert^2_{L^2(0,T;\mathbb{L}^2(\Omega))} \ \forall v\in \mathcal{C}_{\overline{u}}.
		\end{equation*}
	\end{theorem}
	\begin{proof}
		Let $\varepsilon>0$ and $u\in B_\varepsilon(\overline{u}) \cap U_{a,b}$ be given. We define 
		\begin{equation*}
			\rho = \Vert u - \overline{u} \Vert_{L^2(0,T;\mathbb{L}^2(\Omega)) } \leq \varepsilon, \hspace{10pt} v=\frac{1}{\rho}(u-\overline{u}).
		\end{equation*}
		Using the difference $\rho$, we derive
		\begin{align*}
			\mathcal{J}_1 (u) - \mathcal{J}_1 (\overline{u}) &= \mathcal{J}(u) - \mathcal{J}(\overline{u}) + j_1 (u) - j_1 (\overline{u}) \\
			&= \mathcal{J}(u) - \mathcal{J}(\overline{u}) + j_1 (\overline{u}+\rho v) - j_1 (\overline{u}) \\
			&\geq \rho \left[ \mathcal{J}'(\overline{u})v + j_1'(\overline{u})v \right] + \frac{\rho^2}{2} \mathcal{J}''(\overline{u})v^2 + o(\rho^2) .
		\end{align*}
		In case $v \in \mathcal{C}_{\overline{u}}^\tau$, where $\tau$ is defined in Theorem \ref{theo_ssc}, it follows that
		\begin{equation} \label{5.3.5}
			\mathcal{J}_1(u) - \mathcal{J}_1(\overline{u}) \geq \frac{\rho^2}{2} \delta + o(\rho^2).
		\end{equation}
		For the remaining case $v \notin \mathcal{C}_{\overline{u}}^\tau$, one can find that $v$ satisfies the sign conditions \eqref{NUad_1} and \eqref{NUad_2}. Thus $\displaystyle \mathcal{J}'(\overline{u})v+\kappa j_1'(\overline{u})v > \tau$. This deduces that
		\begin{equation} \label{5.3.6}
			\mathcal{J}_1(u) - \mathcal{J}_1(\overline{u}) \geq \rho \tau - \frac{\rho^2}{2} \Vert \mathcal{J}''(\overline{u}) \Vert + o(\rho^2).
		\end{equation} 
		Moreover, since $g'_y(t,\overline{y})$ and $g''_{yy}(t,\overline{y})$ are linear and continuous operators, we get
		\begin{align*}
			\vert g'_y(t,\overline{y})\overline{m}_v \vert &\leq C(t,\overline{y}) \Vert \overline{m}_v \Vert_{W(0,T;H^3,V)},\\
			\vert g''_{yy}(t,\overline{y})\overline{z}_v^2 \vert &\leq C(t,\overline{y}) \Vert \overline{z}_v \Vert^2_{W(0,T;H^3,V)}.
		\end{align*}
		Combining with the fact that $\overline{u} \in U_{a, b}$, both \eqref{5.3.5} and \eqref{5.3.6} imply \eqref{grow_ineq} for sufficiently small $\varepsilon$, thereby completing the proof.
	\end{proof}

	\section{Stability of optimal control problem}
	In this section, we investigate the stability of the optimal solutions with respect to the sparsity parameter $\kappa$. From now on, we shall refer to the three sparse problems collectively as $(\textbf{P}_\kappa)$, unless a distinction is required. \\
	First we consider some properties of the non-sparse problem $(\textbf{P})$. The existence of at least one optimal solutions and the first-order necessary conditions are mentioned above by the authors. For the second-order sufficient optimality conditions, following arguments in \cite{AS2025} with some technical changes, we can establish two equivalent forms as follows
	\begin{equation*}
		(\textbf{SSC}_0)
		\begin{cases}
			\mathcal{J}''(\overline{u})v^2 > 0 \text{ holds for all }v \text{ satisfying conditions} \ \eqref{NUad_1}, \eqref{NUad_2} \ \text{such that } \\
			\mathcal{J}'(\overline{u})v=0.
		\end{cases}
	\end{equation*} 
	and
	\begin{align*}
		(\textbf{SSC})\begin{cases}
			\text{There exist $\eta>0$ and $\mu>0$ such that}\\
			\hspace{30pt} \mathcal{J}''(\overline{u})v^2 \geq \mu \| v \|^2_{L^2(Q)^3}\\
			\text { holds for all } v \ \text{satisfying} \ \eqref{NUad_1}, \eqref{NUad_2}
			\text{ and $\mathcal{J}'(\overline{u})v \leq \eta \|v\|_{L^2(Q)^3}$.}
		\end{cases}
	\end{align*}
	
	We continue to analyze the relationship between the solutions of $(\textbf{P})$ and $(\textbf{P}_{\kappa})$. Hereafter, we denote by $\overline{u}_\kappa$ and $\overline{u}$ the minimizers of problems $(\textbf{P}_\kappa)$ and $(\textbf{P})$, respectively.
	\begin{theorem}\label{theo6.1}
		Let $\{\overline{u}_\kappa\}_{\kappa>0}$ be a family of optimal solutions to problems $(\textbf{P}_\kappa)$. Then every control $\overline{u}$ which is the weak-$^*$ limit of a subsequence $\{\overline{u}_{\kappa_n}\}_{n\geq1} \subset \{u_\kappa\}_{\kappa>0}$ in $L^\infty(Q)^3$  is a minimizer of $(\textbf{P})$. Moreover, we obtain that $\overline{u}_{\kappa_n} \rightarrow \overline{u}$ in $C(\overline{Q})$.
	\end{theorem}
	\begin{proof}
		The existence of  a weakly-$^*$ convergent subsequence $\{\overline{u}_{\kappa_n}\}_{n=1}^{\infty} \subset \{\overline{u}_{\kappa}\}_{\kappa>0}$ is given by the boundedness of $U_{\alpha,\beta}$ in $L^\infty(Q)^3$. Denote by $\overline{y}_{{\kappa_n}}, \overline{y}$ the associated states to $\overline{u}_{\kappa_n}$ and $\overline{u}$, respectively. Using the uniform boundedness of $\{\overline{y}_{{\kappa_n}}\}$ in $C(0,T;D(A))$, we can extract a subsequence $\{\overline{u}_{\kappa_n}\}_{n=1}^{\infty}$ of $\{\overline{u}_\kappa\}$ such that $\{\overline{y}_{{\kappa_n}}\}_{n=1}^{\infty}$ converge weakly to some $\overline{y}$ in $L^\infty(0,T;V)$. Then we can deduce that $\overline{y}=S(\overline{u})$.    
		Take an arbitrary $u \in U_{a,b}$. We find
		\begin{equation*}
			\mathcal{J}(\overline{u}) \leq \liminf_{n \to \infty} \mathcal{J}(\overline{u}_{\kappa_n}) \leq \liminf_{n \to \infty} \mathcal{J}_{\kappa_n} (\overline{u}_{\kappa_n}) \leq \liminf_{n \to \infty} \mathcal{J}_{\kappa_n} (u) = \mathcal{J}(u).
		\end{equation*}
		These inequalities holds for all $u \in U_{a,b}$, which yields the globally optimal of $\overline{u}$.
		Next, we prove that $\left\vert \overline{u}_{\kappa_n} - \overline{u} \right\vert \rightarrow 0$ as $n \rightarrow \infty$. We consider firstly the sparsity term $j_1$.\\
		Since $\overline{u}_{\kappa_n}$ is an optimal solution of problem $(\textbf{P}_\kappa)$, there exist $\overline{\lambda}_{{\kappa_n}} \in W(0,T;H^3,V)$ and $\overline{\zeta}_{\kappa_n}\in \partial j_1(\overline{u}_{\kappa_n})$ such that
		\begin{align}
			\overline{u}_{\kappa_n,i}(x,t) &= \text{Proj}_{[a_i, b_i]} \left\{ -\frac{1}{\gamma} (\overline{\lambda}_{\kappa_n,i}(x,t) + \kappa_n \overline{\zeta}_{\kappa_n,i}(x,t)) \right\}, \label{6.1.1}\\
			\overline{u}_{\kappa_n,i}(x,t) &\Leftrightarrow \vert \overline{\lambda}_{\kappa_n,i}(x,t) \vert \leq \kappa_n, \nonumber \\
			\overline{\zeta}_{\kappa_n,i}(x,t) &=\text{Proj}_{[-1;1]} \left( -\frac{1}{\kappa} \overline{\lambda}_{\kappa_n,i}(x,t) \right) \nonumber.
		\end{align}
		Moreover, the first-order necessary condition with $\overline{u}_{\kappa_n}$ yields
		\begin{equation} \label{6.1.2}
			\begin{aligned}
				\iint_Q u(\overline{\lambda}_{\kappa_n} &+\gamma \overline{u}_{\kappa_n}+ \kappa_n \overline{\zeta}_{\kappa_n}) \, dxdt \\
				&\geq \iint_Q \overline{\lambda}_{\kappa_n} \overline{u}_{\kappa_n} \, dxdt + \gamma \iint_Q \vert \overline{u}_{\kappa_n} \vert^2 dxdt + \iint_Q \kappa_n \overline{u}_{\kappa_n} \overline{\zeta}_{\kappa_n} \, dxdt.
			\end{aligned}
		\end{equation}
		Let $\overline{\lambda}$ be the adjoint state to $\overline{y}$. Then we have that $\overline{\lambda}_{\kappa_n}$ converges weakly to $\overline{\lambda}$ in $W(0,T;H^3,V)$. Thus letting $n\rightarrow \infty$ in \eqref{6.1.2}, together with using the properties of $\overline{\zeta}_{\kappa_n}$ lead to
		\begin{equation*}
			\iint_Q (\overline{\lambda}+ \gamma \overline{u})(u- \overline{u}) \, dxdt \geq 0 \ \text{for all} \ u \in U_{a,b},
		\end{equation*}
		or equivalently,
		\begin{equation} \label{6.1.3}
			\overline{u}_i(x,t) = \text{Proj}_{[a_i,b_i]} \left( -\frac{1}{\gamma} \overline{\lambda}_i(x,t) \right).
		\end{equation}
		Combining \eqref{6.1.1} and \eqref{6.1.3} implies
		\begin{equation*}
			\begin{aligned}
				\vert \overline{u}_{\kappa_n,i}(x,t) - \overline{u}_i(x,t) \vert &\leq \frac{1}{\gamma}\left( \vert \overline{\lambda}_{\kappa_n,i}(x,t) - \overline{\lambda}(x,t) \vert + \kappa_n \vert \overline{\zeta}_{\kappa_n,i} (x,t) \vert \right) \\
				&\leq \frac{1}{\gamma} \left( \Vert \overline{\lambda}_{\kappa_n,i}(x,t) - \overline{\lambda}_i(x,t) \Vert_{C(\overline{Q})} + \kappa_n \right).
			\end{aligned}
		\end{equation*}
		Since $\mathbb{H}^2(\Omega) \hookrightarrow C(\overline{\Omega})$, we have that $W(0,T;H^3,V)$ is compactly embedded in $C(\overline{Q})$. Thus $\overline{\lambda}_{\kappa_n}$ strongly converges to $\overline{\lambda}$ in $C(\overline{Q})$ and we get the final assertion.
	\end{proof}
	\begin{remark}
		\rm{ If the function $j_1$ is replaced by either $j_2$ or $j_3$, our conclusion remains unchanged thanks to the fact that $\overline{\zeta}_{\kappa_n}$ possesses a constant value with each $\overline{u}_{\kappa_n}$ and $\kappa_n \rightarrow0$.}
	\end{remark}
	
	\begin{theorem}\label{theo6.2}
		Let $\overline{u}$ be a strict local minimizer of $(\textbf{P})$ in the $L^2$ - sense, meaning that there exists $\varepsilon>0$ such that
		\[\mathcal{J}(\overline{u}) < \mathcal{J}(u) \ \forall u \in U_{a,b} \cap \overline{B}_{L^2(Q)^3}(\overline{u}, \varepsilon) \setminus\{ \overline{u} \}.\]
		Then there exists a set $\{\overline{u}_{\kappa}\}_{\kappa>0}$ of local minimizers of $(\textbf{P}_{\kappa})$ such that $\overline{u}_\kappa \to \overline{u}$ in $L^2(Q)^3$ as $\kappa \to 0$.
	\end{theorem}
	\begin{proof}
		For a given $\kappa>0$, we consider the problem
		\begin{align*}
			(\mathcal{P}_{\kappa}^{\varepsilon}): \min_{u \in U_{a,b}\cap\overline{B}_{L^2(Q)^3}(\overline{u}, \varepsilon)}  \mathcal{J}_{\kappa} (u)=  \int_0^T [g(t,y_u(t))+ |u(t)|^2] \,dt + \kappa \iint_Q \vert u(x,t) \vert \, dxdt.
		\end{align*}
		The problem $(\mathcal{P}_{\kappa}^{\varepsilon})$ has at least an optimal solution $\overline{u}_\kappa$ for every $\kappa>0$. 
		Set $U_\varepsilon=U_{a,b}\cap\overline{B}_{L^2(Q)^3}(\overline{u}, \varepsilon)$, we also have that $U_\varepsilon$ is bounded in $L^\infty(Q)^3$. Thus we can extract from $\{\overline{u}_\kappa\}$ a subsequence $\{\overline{u}_{\kappa_n} \} _{n=1}^\infty$ that weakly-$^*$ converges to $\widetilde{u}\in U_\varepsilon$. Let $\widetilde{y}$ be the associated state to $\widetilde{u}$. Then
		\begin{equation*}
			\mathcal{J}(\widetilde{u}) \leq \liminf_{n \to \infty} \mathcal{J}(\overline{u}_{\kappa_n}) \leq \liminf_{n \to \infty} \mathcal{J}_{\kappa_n} (\overline{u}_{\kappa_n}) \leq \liminf_{n \to \infty} \mathcal{J}_{\kappa_n} (u) = \mathcal{J}(u).
		\end{equation*}
		for any $u\in U_\varepsilon$, which implies that $\widetilde{u}$ is a local solution to $(\textbf{P})$. Thus $\widetilde{u}$ is a strict local solution and $\widetilde{u}\equiv\overline{u}$. Consequently, there exists $\overline{\lambda}\in W(0,T;H^3,V)$ satisfying the adjoint equations \eqref{adj_eqt} and the following variational inequality
		\begin{equation} \label{6.2.1}
			\iint_Q (\overline{\lambda}+ \gamma \overline{u})(u-\overline{u}) \, dxdt \geq 0, \forall u\in U_\varepsilon.
		\end{equation}
		Furthermore, since $\overline{u}_{\kappa_n}$ is an optimal solution of $(\textbf{P}_\kappa ^\varepsilon)$, we can find $\overline{\zeta}_{\kappa_n}\in \partial j_1(\overline{u})$ and $\overline{\lambda}_{\kappa_n}$ such that
		\begin{equation} \label{6.2.2}
			\iint_Q(\overline{\lambda}_{\kappa_n} + \gamma \overline{u}_{\kappa_n} + \kappa_n \overline{\zeta}_{\kappa_n})(u-\overline{u}_{\kappa_n}) \geq 0, \forall u \in U_\varepsilon.
		\end{equation}
		By the Hilbert projection theorem, \eqref{6.2.1} and \eqref{6.2.2} yield
		\begin{align*}
			\overline{u}&= \text{Proj}_{U_\varepsilon}\left( - \frac{1}{\gamma}\overline{\lambda} \right), \\
			\overline{u}_{\kappa_n} &= \text{Proj}_{U_\varepsilon}\left( -\frac{1}{\gamma} (\overline{\lambda}_{\kappa_n} + \kappa_n \overline{\zeta}_{\kappa_n}) \right).
		\end{align*}
		Hence, 
		\[ \Vert \overline{u}_{\kappa_n} - \overline{u} \Vert_{L^2(Q)^3} \leq \frac{1}{\gamma} \left( \Vert \overline{\lambda}_{\kappa_n} - \overline{\lambda} \Vert_{L^2(Q)^3} + \kappa_n \Vert \overline{\zeta}_{\kappa_n} \Vert_{L^2(Q)^3} \right). \]
		Taking the limit of both sides as $n\rightarrow\infty$, we obtain the desired $L^2$-convergence.
	\end{proof}
	\begin{remark}
		\rm{Analogously to Theorem \ref{theo6.1}, the proofs of Theorem \ref{theo6.1} in case of $j_2$ and $j_3$ only differ from that for $j_1$ only in the definition of $(\mathcal{P}_\kappa^\varepsilon)$.}
	\end{remark}
	
	Let $\overline{u}$ be a strict local optimal solution to problem $(\textbf{P})$. Theorem \ref{theo6.2} ensures the existence of a sequence $\{ \overline{u}_\kappa \}_{\kappa>0}$ of local optimal solutions of the sparse problems $(\textbf{P}_\kappa)$ that converges strongly to $\overline{u}$ in $L^2(Q)^3$. The following theorem investigates the Lipchitz and H\"older stability of the control systems with respect to $j_1,j_2$ and $j_3$. Note that we need not any other conditions of $\overline{u}$ and $\overline{u}_\kappa$ besides the mentioned optimality conditions.
	
	\begin{theorem}\label{theo6.3}
		With the notations above, there exist constants $\varepsilon>0$ and $C>0$ independent of $\overline{u}$ such that it holds
		\begin{enumerate}
			\item Lipchitz stability for $(\textbf{P}_1)$ and $(\textbf{P}_3)$: 
			\begin{equation} \label{lipz_sta_j13}
				\left\Vert \overline{u}_\kappa - \overline{u} \right\Vert_{L^2(Q)^3} \leq C \kappa \ \text{for all} \ 0<\kappa <\varepsilon.
			\end{equation}
			\item H\"older stability for $(\textbf{P}_2)$: 
			\begin{equation} \label{lipz_sta_j2}
				\left\Vert \overline{u}_\kappa - \overline{u} \right\Vert_{L^2(Q)^3} \leq C \kappa^{2/3} \ \text{for all} \ 0<\kappa <\varepsilon.
			\end{equation}
		\end{enumerate}
	\end{theorem}
	
	Before proving this theorem, we establish the following auxiliary result.

	\begin{lemma}\label{lem6.3}
		There exists $\varepsilon >0$ such that for all $\kappa \in (0,\varepsilon)$, the following inequality holds
		\begin{align}\label{6.25}
			\mathcal{J}''(\overline{u} + \theta (\overline{u}_\kappa - \overline{u} ))( \overline{u}_\kappa - \overline{u})^2 \geq \dfrac{\mu}{2}\| \overline{u}_\kappa - \overline{u} \|^2_{L^2(Q)^3}
		\end{align}
		with $\mu$ is given by $(\textbf{SSC})$. 
	\end{lemma}
	
	\begin{proof}
		Let us take $\eta >0$ as in $(\textbf{SSC})$. Then we shall prove that $\overline{u}_\kappa - \overline{u}$ satisfies the conditions of $(\textbf{SSC})$ for sufficiently small $\varepsilon$. To this end, set $v_\kappa = \dfrac{\overline{u}_\kappa - \overline{u}}{\| \overline{u}_\kappa - \overline{u}\|_{L^2(Q)^3}}$. From the boundedness of $v_\kappa$, we can extract a subsequence $\{v_{\kappa_n}\}$ such that $v_{\kappa_n} \rightharpoonup v$ in $L^2(Q)^3$ as $n \to \infty$.  Since the cones $\mathcal{T}_{U_{a,b}} (\overline{u})$ and $\mathcal{T}_{U_{a,b}}(\overline{u}_\kappa)$ are weakly closed, we get $v \in \mathcal{T}_{U_{a,b}} (\overline{u}) \cap - \mathcal{T}_{U_{a,b}}(\overline{u}_\kappa)$. Then we have
		\begin{align*}
			\mathcal{J}'(\overline{u})v \geq 0 \quad \text{ and } \quad \mathcal{J}'(\overline{u}_\kappa)v \leq 0.
		\end{align*}
		Letting $n \to \infty$ yields $\mathcal{J}'(\overline{u})v=0$. From the weak convergence of $v_{\kappa_n}$, we get $\underset{n\rightarrow\infty}{\lim} \mathcal{J}'(\overline{u})v_{\kappa_n} = 0$. Since this happens for all subsequences $\{v_{\kappa_n}\}$ of $\{v_\kappa\}$ which converge weakly in $L^2(Q)^3$, we find $\underset{\kappa \to 0}{\lim} \mathcal{J}'(\overline{u})v_{\kappa} = 0$. Thus, there exists $\varepsilon>0$ such that 
		\begin{align*}
			\mathcal{J}'(\overline{u})v_{\kappa} \leq \eta
		\end{align*}
		holds for all $\kappa \in (0,\varepsilon)$. Thus $\mathcal{J}'(\overline{u})(\overline{u}_\kappa- \overline{u}) \leq \eta \| \overline{u}_\kappa - \overline{u}\|_{L^2(Q)^3}$ for all $\kappa \in (0,\varepsilon)$, which implies that $\mathcal{J}''(\overline{u})(\overline{u}_\kappa - \overline{u})^2 \geq \mu \| \overline{u}_\kappa - \overline{u} \|^2_{L^2(Q)^3}$.
		
		Therefore \eqref{6.25} follows from the facts that $\mathcal{J}(u)$ is of class $C^2$ in $L^2(Q)^3$ and $\overline{u} + \theta (\overline{u}_\kappa - \overline{u}) \to \overline{u}$ in $L^2(Q)^3$. 
	\end{proof}
	\begin{proof}[Proof of Theorem \ref{theo6.3}]
		Since $\overline{u}_\kappa$ is the local solution of $(\textbf{P}_\kappa)$, we have $\mathcal{J}_\kappa(\overline{u}_\kappa)- \mathcal{J}_\kappa(\overline{u}) \leq 0$, which can be written explicitly as
		\begin{equation} \label{6.3_ori}
			\mathcal{J}(\overline{u}_\kappa) - \mathcal{J}(\overline{u}) + \kappa \left(j_i(\overline{u}_{\kappa}) - j_i(\overline{u}) \right) \leq 0, \hspace{20pt} i=1,2,3.
		\end{equation}
		Recalling that \[ \mathcal{J}({\overline{u}_\kappa}) = \int_0^T g(t, y_{\overline{u}_\kappa}(t)) \, dt+ \frac{\gamma}{2} \iint_Q \vert \overline{u}_\kappa(x,t) \vert ^2 \, dxdt, \]
		using Taylor expansion of $\mathcal{J}$ yields
		\begin{equation*}
			\mathcal{J}({\overline{u}_\kappa}) - \mathcal{J}(\overline{u}) = \mathcal{J}'(\overline{u})(\overline{u}_\kappa - \overline{u}) + \frac{1}{2} \mathcal{J}''(\overline{u}+ \theta (\overline{u}_\kappa - \overline{u}))(\overline{u}_\kappa - \overline{u})^2, \theta\in (0,1).
		\end{equation*}
		Since $\overline{u}_\kappa \in U_{a,b}$, the variational inequality \eqref{4.3} and Lemma \ref{lem6.3} imply that
		\begin{equation*}
			\mathcal{J}(\overline{u}_\kappa) - \mathcal{J}(\overline{u}) \geq \frac{\mu}{4} \Vert \overline{u}_k - \overline{u} \Vert_{L^2(Q)^3}^2 .
		\end{equation*}
		Now we consider the difference $j_i(\overline{u})-j_i(\overline{u}_{\kappa})$. \\
		\textbf{Case 1}: $j_i=j_1$.\\
		When $j_i=j_1$, \eqref{6.3_ori} is equivalent to
		\begin{equation*}
			\mathcal{J}(\overline{u}_\kappa) - \mathcal{J}(\overline{u}) + \kappa \iint_Q \left[ \vert \overline{u}_\kappa(x,t) \vert  - \vert \overline{u}(x,t) \vert \right] \, dxdt \leq 0.
		\end{equation*}
		From this we deduce
		\begin{equation*}
			\mathcal{J}(\overline{u}_\kappa) -\mathcal{J}(\overline{u}) \leq \kappa \iint_Q \left[ \vert \overline{u}(x,t) \vert - \vert\overline{u}_\kappa(x,t) \vert \right] \, dxdt \leq \kappa \Vert \overline{u} - \overline{u}_\kappa \Vert_{L^1(Q)^3}.
		\end{equation*}
		Thus
		\begin{equation} \label{6.3.2}
			\frac{\mu}{4} \Vert \overline{u}_k - \overline{u} \Vert_{L^2(Q)^3}^2 \leq \kappa \Vert \overline{u}_\kappa - \overline{u} \Vert_{L^1(Q)^3}.
		\end{equation}
		However, because the domain $Q$ has a finite volume, we get
		\begin{equation} \label{6.3.1}
			\begin{aligned}
				\Vert \overline{u}_\kappa - \overline{u} \Vert_{L^1(Q)^3} &=\iint_Q 1\times \vert \overline{u}_\kappa - \overline{u} \vert \, dxdt \\
				& \leq \left( \iint_Q 1 \, dxdt \right)^{1/2} \left( \iint_Q \vert \overline{u}_\kappa - \overline{u} \vert^2 \, dxdt \right)^{1/2}\\
				&\leq \text{vol}(Q)^{1/2} \Vert \overline{u}_\kappa - \overline{u} \Vert_{L^2(Q)^3}.
			\end{aligned}
		\end{equation}
		Substituting \eqref{6.3.1} into \eqref{6.3.2}, we obtain
		\begin{equation} \label{sta_j1}
			\frac{\mu}{4} \Vert \overline{u}_k - \overline{u} \Vert_{L^2(Q)^3}^2 \leq \kappa \text{vol}(Q)^{1/2} \Vert \overline{u}_\kappa - \overline{u} \Vert_{L^2(Q)^3}.
		\end{equation}
		\textbf{Case 2}: $j_i=j_2$.\\
		Since $\displaystyle j_2(u)=\Vert u\Vert_{L^2(0,T;\mathbb{L}^1(\Omega))}=\left( \int_0^T \Vert u(t) \Vert_{\mathbb{L}^1(\Omega)}^2 \, dt \right)^{1/2}$, \eqref{6.3_ori} can be rewritten as
		\begin{equation} \label{6.3.4}
			\begin{aligned}
				\mathcal{J}(\overline{u}_{\kappa}) - \mathcal{J}(\overline{u}) &\leq \kappa \left[ j_2(\overline{u}) - j_2 (\overline{u}_{\kappa}) \right]\\
				&= \kappa \left( \int_0^T \Vert \overline{u}(t) \Vert_{\mathbb{L}^1(\Omega)}^2 \, dt \right)^{1/2} - \kappa \left( \int_0^T \Vert \overline{u}_{\kappa}(t) \Vert_{\mathbb{L}^1(\Omega)}^2 \, dt \right)^{1/2}.
			\end{aligned}
		\end{equation}
		Notice that $\overline{u}$ is an optimal solution of problem $(\textbf{P})$, meaning that $\mathcal{J}(\overline{u}_{\kappa}) - \mathcal{J}(\overline{u})\geq 0$, we get $j_2(\overline{u}) \geq j_2(\overline{u}_{\kappa})$. Thus we can write
		\begin{align*}
			&\left( \int_0^T \Vert \overline{u}(t) \Vert_{\mathbb{L}^1(\Omega)}^2 \, dt \right)^{1/2} - \left( \int_0^T \Vert \overline{u}_{\kappa}(t) \Vert_{\mathbb{L}^1(\Omega)}^2 \, dt \right)^{1/2} \\
			&\leq \left( \int_0^T \Vert \overline{u}(t) \Vert_{\mathbb{L}^1(\Omega)}^2 \, dt - \int_0^T \Vert \overline{u}_{\kappa}(t) \Vert_{\mathbb{L}^1(\Omega)}^2 \, dt \right)^{1/2} \\
			&= \left[\int_0^T  \left[  \left( \int_{\Omega} |\overline{u}(t)| \, dx \right)^2 - \left(\int_{\Omega} \vert \overline{u}_{\kappa}(t) \vert \, dx \right)^2\right]dt \right]^{1/2}.
		\end{align*}
		Moreover, it follows
		\begin{align*}
			\left( \int_{\Omega} |\overline{u}(t)| \, dx \right)^2 &- \left(\int_{\Omega} \vert \overline{u}_{\kappa}(t) \vert \, dx \right)^2 \\
			&= \left( \int_{\Omega} (\vert \overline{u}(t) \vert + \vert \overline{u}_{\kappa}(t) \vert) \, dx \right) \left( \int_{\Omega} (\vert \overline{u}(t) \vert - \vert \overline{u}_{\kappa}(t) \vert) \, dx \right) \\
			&\leq C(b). \int_{\Omega} \vert \overline{u}(t) - \overline{u}_{\kappa}(t)\vert \, dx
		\end{align*}
		provided that $\Omega$ is a bounded domain.
		Hence, 
		\begin{equation} \label{6.3.5}
			\begin{aligned}
				\int_0^T  \left[  \left( \int_{\Omega} |\overline{u}(t)| \, dx \right)^2 - \left(\int_{\Omega} \vert \overline{u}_{\kappa}(t) \vert \, dx \right)^2\right]dt &\le \int_0^T C(b). \int_{\Omega} \vert \overline{u}(t) - \overline{u}_{\kappa}(t)\vert \, dt \\
				&\le C(b,T).\Vert \overline{u}_{\kappa} - \overline{u}\Vert_{L^1(Q)^3}. 
			\end{aligned}
		\end{equation}
		Substituting \eqref{6.3.5} into \eqref{6.3.4} leads to
		\begin{equation}
			\mathcal{J}(\overline{u}_{\kappa}) - \mathcal{J}(\overline{u}) \leq \kappa C(b,T)^{1/2}.\Vert \overline{u}_{\kappa} - \overline{u}\Vert_{L^1(Q)^3}^{1/2}.
		\end{equation}
		Then by using Lemma \ref{lem6.3} and \eqref{6.3.1}, we obtain
		\begin{align*}
			\frac{\mu}{4} \Vert \overline{u}_k - \overline{u} \Vert_{L^2(Q)^3}^2 &\leq \kappa C(b,T)^{1/2}.\Vert \overline{u}_{\kappa} - \overline{u}\Vert_{L^1(Q)^3}^{1/2} \\
			&\leq \kappa C(b, T)^{1/2} \text{vol}(Q)^{1/4} \Vert \overline{u}_{\kappa_n}- \overline{u} \Vert_{L^2(Q)^3}^{1/2}.
		\end{align*}
		Consequently,
		\begin{equation} \label{sta_j2}
			\Vert \overline{u}_{\kappa} - \overline{u} \Vert_{L^2(Q)^3} \leq \left(\frac{4}{\mu}\right)^{2/3}C(b,T)^{1/3} \text{vol}(Q)^{1/6} \kappa^{2/3}.
		\end{equation}
		\textbf{Case 3}: $j_i=j_3$.\\
		Analogously to Case 2, we also get
		\begin{equation} \label{6.3.6}
			\begin{aligned}
				\mathcal{J}({\overline{u}}_{\kappa}) - \mathcal{J}(\overline{u}) &\leq \kappa \left[ j_3(\overline{u}) - j_3 (\overline{u}_{\kappa}) \right] \\
				&= \kappa\int_\Omega \Vert \overline{u}(x) \Vert_{L^2(0,T)^3}\, dx - \kappa \int_\Omega \Vert \overline{u}_{\kappa} (x)\Vert_{L^2(0,T)^3} \, dx.
			\end{aligned}
		\end{equation}
		The definition of $j_3$ yields
		\begin{equation*}
			\mathcal{J}({\overline{u}}_{\kappa}) - \mathcal{J}(\overline{u}) \leq \kappa \int_\Omega\left[ \left( \int_0^T \vert \overline{u}(x,t) \vert^2 \, dt\right)^{1/2} - \left( \int_0^T \vert \overline{u}_{\kappa}(x,t) \vert^2 \, dt\right)^{1/2} \right]dx.
		\end{equation*}
		On the other hand, we can write
		\begin{align*} 
			&\left( \int_0^T \vert \overline{u}(x,t) \vert^2 \, dt\right)^{1/2} - \left( \int_0^T \vert \overline{u}_{\kappa}(x,t) \vert^2 \, dt\right)^{1/2} \\&= \displaystyle \frac{\int_0^T \vert \overline{u}(x,t) \vert^2 \, dt - \int_0^T \vert \overline{u}_{\kappa}(x,t) \vert^2 \, dt }{\left( \int_0^T \vert \overline{u}(x,t) \vert^2 \, dt\right)^{1/2} + \left( \int_0^T \vert \overline{u}_{\kappa}(x,t) \vert^2 \, dt\right)^{1/2}} \\
			&\leq \frac{\int_0^T (\vert \overline{u}(x,t) \vert^2 - \vert\overline{u}_{\kappa}(x,t)\vert^2) \, dt}{2a\sqrt{T}}.
		\end{align*}
		It follows that
		\begin{align*}
			\int_0^T \left( \vert \overline{u}(x,t) \vert^2 - \vert\overline{u}_{\kappa}(x,t)\vert^2 \right) \, dt &= \int_0^T \left[ (\vert \overline{u}(x,t)\vert+\vert \overline{u}_{\kappa}(x,t)\vert)(\vert \overline{u}(x,t)\vert-\vert \overline{u}_{\kappa}(x,t)\vert) \right] \, dt \\
			&\leq C(T,b)\int_0^T \vert \overline{u}(x,t) - \overline{u}_{\kappa}(x,t) \vert \, dt.
		\end{align*}
		Therefore
		\begin{align*}
			&\left( \int_0^T \vert \overline{u}(x,t) \vert^2 \, dt\right)^{1/2} - \left( \int_0^T \vert \overline{u}_{\kappa}(x,t) \vert^2 \, dt\right)^{1/2} \\
			&\leq \frac{C(T,b)}{2a\sqrt{T}} \int_0^T \vert \overline{u}(x,t) - \overline{u}_{\kappa}(x,t) \vert \, dt.
		\end{align*}
		Finally,
		\begin{align*}
			\mathcal{J}({\overline{u}}_{\kappa}) - \mathcal{J}(\overline{u}) &\leq \kappa \int_\Omega \frac{C(T,b)}{2a\sqrt{T}} \int_0^T \vert \overline{u}(x,t) - \overline{u}_{\kappa}(x,t) \vert \, dt \\
			&\leq \kappa C(T,a,b,\vert\Omega\vert) \Vert \overline{u}_{\kappa} - \overline{u} \Vert_{L^1(Q)^3}.
		\end{align*}
		Applying Lemma \ref{lem6.3} and \eqref{6.3.1} once again, we conclude
		\begin{align*}
			\frac{\mu}{4} \Vert \overline{u}_k - \overline{u} \Vert_{L^2(Q)^3}^2 &\leq \kappa C(T,a,b,\vert\Omega\vert).\Vert \overline{u}_{\kappa} - \overline{u}\Vert_{L^1(Q)^3} \\
			&\leq \kappa C(T,a,b,\vert\Omega\vert) \text{vol}(Q)^{1/2} \Vert \overline{u}_{\kappa_n}- \overline{u} \Vert_{L^2(Q)^3}.
		\end{align*}
		Consequently,
		\begin{equation} \label{sta_j3}
			\Vert \overline{u}_{\kappa} - \overline{u} \Vert_{L^2(Q)^3} \leq \frac{4}{\mu}C(T,a,b,\vert\Omega\vert) \text{vol}(Q)^{1/2} \kappa.
		\end{equation}
		It follows directly from \eqref{sta_j1}, \eqref{sta_j2} and \eqref{sta_j3} that \eqref{lipz_sta_j13} and \eqref{lipz_sta_j2} hold. We achieve the Lipschitz and H\"older convergence of the solutions to the sparse optimal control problems toward the corresponding standard non-sparse ones.
	\end{proof}
	\begin{remark}
		\rm{ The stability of the sparse control systems \eqref{lipz_sta_j13} and \eqref{lipz_sta_j2} can also be presented under the $L^1$-norm. }
	\end{remark}
	\noindent{\bf Acknowledgements.} 
	
	
	
\end{document}